\documentclass{amsart}

\usepackage{amsmath}
\usepackage{amsthm}
\usepackage{amssymb}
\usepackage{amsfonts}

\newtheorem{theorem}{Theorem}[section]
\newtheorem{lemma}[theorem]{Lemma}
\newtheorem{proposition}[theorem]{Proposition}
\newtheorem{corollary}[theorem]{Corollary}

\theoremstyle{definition}

\theoremstyle{remark}
\newtheorem{remark}[theorem]{Remark}

\numberwithin{equation}{section}

\newcommand{\ca}{{\mathcal A}}    \newcommand{\cb}{{\mathcal B}}   
\newcommand{\ce}{{\mathcal E}}    \newcommand{\cf}{{\mathcal F}}   \newcommand{\cg}{{\mathcal G}}
    \newcommand{\ck}{{\mathcal K}}   \newcommand{\cl}{{\mathcal L}}
\newcommand{\cm}{{\mathcal M}}       
  \newcommand{\cs}{{\mathcal S}}   \newcommand{\ct}{{\mathcal T}}
\newcommand{\cu}{{\mathcal U}}       
\newcommand{\cw}{{\mathcal W}}

\newcommand{\R}{\mathbb{R}}   \newcommand{\N}{\mathbb{N}}

\newcommand{\vf}{{\varphi}}

\newcommand{\ex}{\mbox{\sf Exc}({\mathcal U})}
\newcommand{\pot}{\mbox{\sf Pot}({\mathcal U})}

\newcommand{\potb}{\mbox{\sf Pot}({\mathcal U}_\beta)}
\newcommand{\exb}{\mbox{\sf Exc}({\mathcal U}_{\beta})}

\newcommand*{\Fscr}{\mathcal F}

\renewcommand*{\d}{\;\mathrm{d}}

\newcommand{\ds}{\displaystyle}

\begin{document}

\vspace*{-40mm}

\noindent
To appear in {\it Transactions of the American Mathematical Society}\\[40mm]

\title{\mbox{Measure-valued discrete branching Markov processes}}

\author{Lucian Beznea}
\address{Simion Stoilow Institute of Mathematics  of the Romanian Academy, Research unit No. 2,  
P.O. Box \mbox{1-764,} RO-014700 Bucharest, Romania, and University of Bucharest, Faculty of Mathematics and Computer Science}
\email{lucian.beznea@imar.ro}
\thanks{This work was supported by a grant of the Romanian National Authority for Scientific
Research, CNCS --UEFISCDI, project number PN-II-ID-PCE-2011-3-0045. For the second author the research was financed  through the project "Excellence Research Fellowships for Young Researchers", the 2015 Competition, founded by the Research Institute of the University of Bucharest (ICUB) }

\author{Oana Lupa\c scu}
\address{Simion Stoilow Institute of Mathematics of the Romanian Academy,
Research group of the POSDRU Project 82514. 
{\it Current address:}  {Institute of Mathematical Statistics and Applied Mathematics of the Romanian Academy,  Calea 13 Septembrie 13, Bucharest, Romania, and the  Research Institute of the University of Bucharest (ICUB)}}
\email{oana.lupascu@yahoo.com}

\subjclass[2000]{Primary 60J80, 60J45, 60J35; Secondary 60J40, 47D07}




\keywords{Discrete branching process, measure-valued process, branching kernel, branching semigroup, excessive function, standard process,
compact Lyapunov function}

\begin{abstract}
We construct and study  branching Markov processes on the space of finite configurations of the state space of a 
given standard process, controlled by a branching    kernel and a  killing one.
In particular,  we may start with a superprocess, obtaining  a branching  process with state space  the finite configurations of positive finite measures on a  topological space.
A main tool in proving the path regularity of the branching process is the existence  of  convenient
superharmonic functions having compact level sets, allowing the use of appropriate potential theoretical methods.
\end{abstract}

\maketitle

\section{Introduction}\label{1}
The description of a discrete branching process is as follows (cf. e.g.,  \cite{Si68}, page 235,
and \cite{DaGoLi02}). 
An initial particle starts at a point of   a set $E$ and moves according to a standard Markov process with state space $E$  (called base process)
until a random terminal time when it is destroyed and replaced by a finite number of new particles, its direct descendants. 
Each direct descendant moves according to the same right standard process  
until its own terminal time when it too is destroyed and replaced by second generation particles and the process continues in this manner. 
N. Ikeda,  M. Nagasawa, S. Watanabe, and M. L. Silverstein (cf. \cite{INW68}, \cite{INW69}, and \cite{Si68})  
indicated the natural connection between discrete branching processes and nonlinear partial differential operators $\Lambda$ of the type
\[
\Lambda u :=\mathcal{L} u + \sum_{k=1}^{\infty} q_k u^k,
\]
where $\mathcal{L}$ is the infinitesimal generator of  the given base process  
and the coefficients $q_k$  are positive, Borelian functions with $ \sum_{k=1}^{\infty} q_k=1$.
In this case each direct descendant starts at the terminal position of the parent particle and $q_k(x)$
is the probability that a particle destroyed at $x$ has precisely $k$ descendants. 
It is possible to consider a more general nonlinear part for the above operator, generated by a branching kernel $B$; 
the descendants start to move from their birthplaces  which have been distributed according to $B$.
Thus, these processes are also called "non-local branching" (cf. \cite{DaGoLi02});
from the literature about  branching processes we indicate the classical monographs 
\cite{Ha63}, \cite{AtNe72},   \cite{AsHe83},  the recent one \cite{Li11}, and the lecture notes \cite{LeG99} and \cite{DaPe12}.

In this paper we construct  discrete branching Markov processes associated to operators of the type $\Lambda$,
using several  analytic and probabilistic potential theoretical methods.
The base space of the process is the set $\widehat{E}$ of all finite configurations of $E$.

The structure and main results of this paper are the following.
In Section 2 we collect some preliminaries on the resolvents of kernels and basic notions of potential theory. 
We present in $(2.1)$ and Lemma \ref{lem-quasileft} a suitable result on the existence of a c\`adl\`ag, quasi-left continuous strong Markov processes, 
given  a resolvent of kernels,  imposing conditions on the resolvent. The branching kernels on the space of finite configurations
are introduced in Section 3.

Section 4 is devoted to the construction of the measure-valued discrete branching processes.
The first  step is to solve the nonlinear evolution equation induced by $\Lambda$ (see Proposition \ref{prop4.1} and Remark \ref{rem4.2} $(ii)$ below).
Then, using a technique of absolutely monotonic operators developed in \cite{Si68}, it is possible to 
construct (cf. Corollary \ref{cor4.3}) a Markovian transition function on $\widehat{E}$, formed by branching kernels.
We follow the classical approach  from  \cite{INW68} and \cite{Si68}, 
but we consider  a more general  frame,  the  given topological  space $E$ being  a Lusin one  and not  more  compact
(see Ch. 5 in \cite{AsHe83} for the locally compact space situation; a detailed comment is given in Remark \ref{rem4.11} $(i)$).

The second step is to show that  the transition function we constructed on $\widehat{E}$ is indeed associated to a standard process with state space $\widehat{E}$. 
The main result is Theorem \ref{thm4.7},
its proof involves the entrance space $\widehat{E}_1$, an extension of $\widehat{E}$ constructed by using  a Ray type compactification method.
We apply the mentioned results from Section 2, 
showing  that the required imposed conditions from $(2.1)$ are satisfied by the resolvent of kernels on $\widehat{E}$ associated with
the branching semigroup constructed in the previous step.
In Proposition \ref{prop4.7} (ii) we emphasize  relations between a class of excessive functions with respect to the base process $X$ 
and two classes of excessive functions (defined on  $\widehat{E}$)  with respect to the forthcoming branching process:
the linear and the exponential type excessive functions. 
A particular linear excessive function for the branching process becomes a function having compact level sets 
(called compact Lyapunov function) and will lead to the tightness property of the capacity on $\widehat{E}$ (see Proposition \ref{prop4.7} (iii)). 
It turns out that it is necessary to make a perturbation of $\mathcal{L}$ with a kernel induced by the given branching kernel, and we present it in 
Proposition \ref{prop4.4}. 

The above mentioned  tools were useful   in the case of  the continuous branching processes too (cf. \cite{Be11} and \cite{BeLuOp12}),  
e.g.,  for the  superprocesses in the sense of E. B. Dynkin 
(cf. \cite{Dy02}; see Section 5 for the basic  definitions),  
like the super-Brownian motion,
processes on the space of all finite measures on $E$ induced by
operators of the form $\mathcal{L}u-u^\alpha$ with $1<\alpha \leqslant  2$.
We establish in Remark \ref{rem4.4} several  links  with the nonlinear partial differential equations associated with the branching semigroups
and we point out connections between the continuous and discrete branching processes.
Note that  a cumulant  semigroup (similar to the continuous branching case; see $(5.2)$  below) 
is introduced in Corollary \ref{cor4.3} for the discrete branching processes.
In particular,  when the base process $X$ is the Brownian motion, 
then the cumulant semigroup of the induced discrete branching process  formally satisfies a nonlinear evolution equation involving the square of the gradient.
Finally,  recall that the method of finding a convenient compact Lyapunov function was originally  
applied in order to obtain martingale solutions for singular stochastic differential equations on Hilbert spaces  (cf. \cite{BeBoRo06a} and \cite{BeRo11b}) 
and for proving the standardness property of infinite dimensional L\'evy processes on Hilbert spaces (see \cite{BeCoRo11}). 

We complete the main result  with an application as suggested  in \cite{Ha63}, page 50, 
where T. E. Harris emphasizes the interest for branching processes 
for which "each object is a very complicated entity; e.g., an object may itself be a population". 
More precisely, because we  may consider base processes  with general state space, 
it might be a continuous branching process  playing this role. 
In Section 5, Corollary \ref{cor:cont-disc}, we obtain in this way a branching Markov process, 
having the space of finite configurations of positive finite measures on $E$ as base space. 
Note that in \cite{BeLeGLeJ97} new branching processes are generated starting with  a superprocess and using an appropriate  subordination theory.

The proofs of  Lemma \ref{lem-quasileft}, Proposition \ref{prop4.1},  and Proposition \ref{prop4.4} are presented in Appendix, 
which includes as well complements on  excessive measures, Ray cones, and the extension of the base space $E$ up to the entrance space $E_1$, using the energy functional.

The first named author acknowledges enlightening discussions with P. J. Fitzsimmons and J.-F. Le Gall, during the preparation of this paper.

\section{Preliminaries on the resolvents of kernels and standard\\ processes}  

Let $E$ be  a Lusin topological  space (i.e.,  $E$ is homeomorphic to a Borel subset  of a
 compact  metric space)  with Borel $\sigma$-algebra $\cb(E)$.
Let further $\mathcal{U}=(U_\alpha)_{\alpha>0}$ be a sub-Markovian resolvent of kernels on $(E,\cb(E))$. 

We denote by $\mathcal{E}(\mathcal{U})$  the set of all $\cb(E)$-measurable $\mathcal{U}$-\textit{excessive functions}: 
$u\in\mathcal{E}(\mathcal{U})$
  if and only if $u$ is a non-negative, numerical, $\mathcal{B}(E)$-measurable function, $\alpha U_{\alpha} u\leqslant  u
$ for all $\alpha>0$, and  $\lim_{\alpha\rightarrow \infty}\alpha U_\alpha u(x)=u(x)$, $x\in E.$ 

If $\beta>0,$ we denote by $\mathcal{U}_\beta$ the sub-Markovian resolvent of kernels 
$\cu_\beta=(U_{\beta+\alpha})_{\alpha>0}.$  
If $w$ is a $\mathcal{U}_\beta$\textit{-supermedian function} (i.e., $\alpha U_{\beta+\alpha} w\leqslant  w
$ for all $\alpha>0$) then its $\mathcal{U}_\beta$-excessive regularization $\widehat{w}$ is given by 
$\widehat{w}(x)=\sup_{\alpha}\alpha U_{\beta+\alpha}w(x)$,  $x\in E$.
Let $\cs({\cu_\beta})$  denote the set of all $\cb(E)$-measurable $\cu_\beta$-supermedian  functions.

For a familiy $\mathcal{G}$ of real-valued functions on $E$ we denote by $\sigma(\mathcal{G})$ 
(resp. by $\ct(\cg)$)
the $\sigma-$algebra   (resp. the topology) generated by $\mathcal{G}$
and by $b\mathcal{G}$ (resp. $[\cg],$   $\overline{\cg}$) the subfamily of bounded functions
from  $\mathcal{G}$  (resp. the linear space spanned by $\cg$, the closure in the supremum norm of $\cg$).

We denote by $p\cb(E)$ the set of all positive $\cb(E)$-measurable functions on $E$.

\begin{remark} \label{rem2.1}  
The vector space $[b\ce(\cu_\beta)]$ does not depend on $\beta>0$.
\end{remark}

\vspace{-1.6mm}

\noindent
Indeed, if $\alpha>0$, $\alpha<\beta$, then because $\ce(\cu_\alpha)\subset \ce(\cu_\beta)$ it is sufficient to prove that
$b\ce(\cu_\beta)\subset  [b\ce(\cu_\alpha)]$. For,  if $v\in b\ce(\cu_\beta)$ then both $U_\alpha v$ and
$v+(\beta-\alpha)U_\alpha v$ belong to $b\ce(\cu_\alpha)$.

We assume further in this section that  for some $\beta >0$ there exists a strictly positive constant $k$ with
$k \leqslant  U_\beta 1$ (in particular, this happens if the resolvent  $\cu$ is Markovian, i.e., $\alpha U_\alpha 1=1$).

Recall that  a right process $X=(\Omega,\cf,\cf_t, X_t,\theta_t, P^x)$  with state space $E$ is called 
 \textit{standard} if for every finite measure $\mu$ on $(E, \cb(E))$ $X$ has \textit{c\`adl\`ag trajectories} under $P^\mu$,  i.e.,  
 it possesses left limits in $E$ $P^\mu$-a.e. on $[0,\zeta)$ and $X$ is 
 \textit{quasi-left continuous up to $\zeta$} $P^\mu$-a.e., i.e., for every increasing sequence $(T_n)_n$ of stopping times with 
 $T_n\nearrow T$ we have $X_{T_n}\longrightarrow X_T\;P^\mu$-a.e. on $[T<\zeta]$, 
 $\zeta$ being the life time of $X$;
 a \emph{stopping time} is a map $T: \Omega\longrightarrow \overline{\R}_+$
such that the set $[T\leqslant  t]$ belongs to $\Fscr_t$ for all
$t\geqslant  0$.
 \\

The next result is the convenient  one for the construction  of the discrete branching measure-valued processes 
we give in Theorem  \ref{thm4.7} below,  it follows from \cite{BeLuOp12}, Theorem 2.1,  and 
it  is a consequence of \cite{BeRo11b}, Theorem 5.2, Corollary 5.3 $(ii)$, and Theorem 5.5 $(i)$. 
\\

\noindent
$(2.1)\quad$ 
Suppose that  the following three conditions  are satisfied by  
the sub-Markovian resolvent of kernels $\, \mathcal{U}=(U_\alpha)_{\alpha>0}$ on $(E, \cb(E))$:\\

$(h1)\quad$  $\sigma(\mathcal{E}(\mathcal{U}_{\beta}))=\cb(E)$  and all the points  of  $E$  
are non-branch points with respect to  $\mathcal{U}_\beta,$
that is $1\in\mathcal{E}(\mathcal{U}_\beta)$ and if $u,v\in \mathcal{E}(\mathcal{U}_\beta)$
 then for all $x\in E$ we have $\inf(u,v)(x)=\widehat{\inf(u,v)(x)}.$
 
 ${(h2)}\quad$  For every $x\in E$ there exists $v_x\in \mathcal{E}(\mathcal{U}_\beta)$ 
such that $v_x(x)<\infty$ and the set $[v_x\leqslant  n]$ is relatively compact for all $n$; 
such a function $v_x$ is called \textit{compact Lyapunov function}.

$(h3)$ There exists a countable subset $\cf$ of $[b\ce(\cu_\beta)]$ generating the topology of $E$, $1\in \cf$, 
and there exists $u_o\in \ce(\cu_\beta)$, $u_o<\infty$, 
such if $\xi, \eta$ are two finite $\cu_\beta$-excessive measures
with $L_\beta (\xi+\eta, u_o)<\infty$ and such that $L_\beta(\xi, \vf)=L_\beta(\eta, \vf)$ for all $\vf\in \cf$, then $\xi=\eta$; 
here $L_\beta$ denotes  the energy functional associated with $\cu_\beta$, see $(A1)$ in Appendix.

Then there exists a c\`adl\`ag process $X$ with state space $E$ such that $\mathcal{U}$ is the resolvent of $X$, i.e.,
for all $\alpha>0$, $f\in bp\cb(E)$, and $x\in E$ we have $U_\alpha f(x)= E^x\int_0^\infty e^{-\alpha t} f(X_t) dt$.\\

\begin{remark} \label{rem2.2} 
$(i)$ Condition $(h1)$ is  necessary in order to deduce that $\cu$ is the resolvent family of a (Borel) right process;
see  \cite{Sh88}, \cite{St89}, \cite{BeBo04}, and \cite{BeBoRo06b}.

$(ii)$  According to \cite{LyRo92} and \cite{BeBo05}, 
condition $(h2)$ is necessary for proving that the process $X$ has c\`adl\`ag trajectories. 
It is related to the tightness of the associated capacity,  we give some details below.

$(iii)$ The quasi-left  continuity of the forthcoming measure-valued branching process will be deduced from the next lemma.
Some arguments in its proof are classical, e.g., similar to  the Ray resolvent case 
(see for example  Theorem (9.21) from \cite{Sh88} , the proof of Lemma IV.3.21 from
\cite{MaRo92}, and the proof of Theorem 3.7.7 from \cite{BeBo04}).
However, none of the existing results covers our context, therefore for the reader convenience, 
we give the proof of the lemma in Appendix $(A2)$.
\end{remark}

\begin{lemma} \label{lem-quasileft} 
Let $X$ be a  right process with state space $E$ and c\`adl\`ag trajectories.
Let $(p_t)_{t\geqslant  0}$ be its transition function, assume that
there exists a countable subset $\cf$ of $[b\ce(\cu_\beta)]$ generating the topology of $E$
and a family $\ck\subset \overline{[\cf]}$  
which is multiplicative  (i.e., if $f,g\in \ck$ then  $fg\in \ck$) and separating the points of $E$. 
Then the process $X$ is quasi-left continuous, hence standard, provided that
$p_tf$ belongs to $\overline{[\cf]}$
for all $f\in \ck$ and $t>0$. 
If $(p_t)_{t\geqslant  0}$ is Markovian then it is enough to assume that
$\ck$ is a family of bounded,  continuous, real-valued functions on $E$ which is multiplicative, separating the points of $E$, and
$p_tf$ is a continuous function for all $f\in \ck$ and $t>0$. 
\end{lemma}

We present some necessary preliminaries on the tightness  of the capacity induced by a sub-Markovian resovent of kernels.

Assume that  $\mathcal{U}=(U_\alpha)_{\alpha>0}$ is a sub-Markovian resolvent of kernels 
on $(E, \cb(E))$ satisfying condition $(h1)$. 
If $M\in \cb(E)$ and $u\in \mathcal{E}(\mathcal{U}_\beta)$, then the 
 \textit{reduced function } (with respect to $\mathcal{U}_\beta$) \textit{of u on M} is the function $R_\beta^M u$ defined by
\[
R_\beta^M u:=\inf\{v\in\mathcal{E}(\mathcal{U}_\beta): v\geqslant  u\mbox{ on }M\}.
\]
The reduced function $R_\beta^M u$ is a universally $\cb(E)$-measurable 
$\mathcal{U}_\beta$-supermedian function.

Let further $u_o:=U_\beta f_o,$ where $f_o$ is a bounded, strictly positive $\cb(E)$-measurable function,
and fix a finite measure $\lambda$ on $(E, \cb(E))$.
Consider the functional $M\longmapsto c_\lambda(M),M\subset E$, defined as
\[
c_\lambda(M):=\inf\{\lambda(R_\beta^G u_o): G \mbox{ open}, M\subset G\}.
\]
 By \cite{BeBo04}  $c_\lambda$ is a Choquet capacity on $E$.\\

\noindent
$(2.2)\quad$ The following assertions are equivalent  (see Proposition 4.1 in \cite{BeRo11b} and Proposition 2.1.1 in \cite{BeRo11a}):

$(i)$  The capacity $c_\lambda$ is \textit{tight}, i.e., there exists an increasing sequence $(K_n)_n$ of compact sets such that 
$\inf_n c_\lambda(E \setminus K_n)=0.$

$(ii)$ There exist a  $\mathcal{U}_\beta$-excessive function $v$  which is finite
 $\lambda$-a.e.  and a bounded strictly positive $\mathcal{U}_\beta$-excessive function $u$ such that the set
$[\frac{v}{u}\leqslant  \alpha]$ is relatively compact for all $\alpha>0$. 

$(iii)$  There exists a $\mathcal{U}_\beta$-excessive function $v$ which is $\lambda$-integrable and 
such that the set $[\frac{\!\! v}{u_o}\leqslant  \alpha]$ is relatively compact for all $\alpha>0$. 

\begin{remark} \label{rem2.4} 
(i) If there exists a strictly positive constant k such that $k\leqslant  u_o$ 
(in particular, this happens if the resolvent $\mathcal {U}$  is Markovian), 
then in the above assertion $(ii)$ one can take  $u = 1$.

(ii) If $\cu$ is the resolvent of a right process $X$, then the
following fundamental result of G. A. Hunt holds for all $A\in
\cb(E)$, $x\in E$,  and $u\in \ce(\cu_\beta)$:
\[
R_\beta^A u(x)= E^x(e^{-\beta D_A} u (X_{D_A})),
\]
where $D_A$ is the entry time of $A$, $D_A:=\inf \{ t\geqslant  0 :$ $ X_t\in A \}$;
 see e.g. \cite{DeMe75}.
Consequently, the capacity $c_\lambda$ on $E$
is {tight} if and only if  
$P^\lambda (\lim_n D_{E\setminus K_n} < \zeta )=0.$
\end{remark}

\section{Branching kernels on the space of finite configurations} 

The state space for the forthcoming  discrete branching process will be 
the set  $\widehat{E}$ of  finite sums of Dirac measures on $E$, defined as
$$
\widehat{E}:=\left\{\sum_{k\leq k_0}\delta_{x_k}:k_{ 0}\in\N^*, x_k\in E\textrm{ for all } 1\leq k\leq k_0\right\}\cup\{{\bf 0}\},
$$
where ${\bf 0}$ denotes the zero measure. 
The set  $\widehat{E}$   is identified   with the union of all symmetric $m$-th powers
$E^{(m)}$ of $E$
(i.e.,  if $m\geq 1$ then $E^{(m)}$ is the factorization of the Cartesian product $E^m$  by the equivalence relation induced by the
permutation group $\sigma^m$; see, e.g., \cite{BeOp11} for details), 
hence
$$
\widehat{E}= \bigcup_{m \geq 0}E^{(m)},
$$
where $E^{(0)}:=\{\bf 0\}$ (see, e.g., \cite{INW68}).
The set $\widehat{E}$ is called the {\it space of finite configurations of E} 
and it is endowed with the topology of disjoint union of topological spaces 
and the corresponding Borel $\sigma$-algebra $\cb(\widehat{E})$; see \cite{FiKoOl11}.

Let $M(E)$ be the set of all positive finite measures on $E$, 
endowed with the weak topology and let $\cm(E)$ be its Borel $\sigma$-algebra.
For a function $f\in p\cb(E)$ we consider the mappings  $l_f:
M(E)\longrightarrow {\overline{\R}}_+$  and $e_f: M(E)\longrightarrow [0,1]$, defined as
$$
l_f(\mu):= \langle \mu,f \rangle:=\int_E f \mbox{d}\mu, \  \mu\in M(E),
\quad e_f:=\exp\ (-l_f).
$$
Note that the Borel  $\sigma$-algebra $\cm(E)$  of  $M(E)$ is  generated by
$\{l_f \, : \, f \in {bp}{\mathcal B}(E) \}$, 
$\widehat{E}$ becomes a Borel  subset of $M(E)$,  
and the trace of ${\mathcal M}(E)$ on $\widehat{E}$ is $\cb({\widehat{E}})$.

Recall that if $p_1, p_2$  are two finite measures on $\widehat{E}$, then
their convolution $p_1 *p_2$ is the finite measure on $\widehat{E}$
defined for every $F\in bp\cb(\widehat{E})$ by
\[
\int_{\widehat{E}} p_1*p_2(\mbox{d}\nu)F(\nu) := \int_{\widehat{E}} p_1(\mbox{d}\nu_1) \int_{\widehat{E}} p_2(\mbox{d}\nu_2)
F(\nu_1+\nu_2).
\]
In particular, if $f\in p\cb(E)$ then
$p_1  * p_2 (e_f)=p_1 (e_f) p_2 (e_f). $

According with \cite{Si68}, a kernel $N$ on $(\widehat{E},\cb(\widehat{E}))$
 which is sub-Markovian (i.e., $N1\leqslant  1$) is
called \textit{branching kernel} provided that for all $\mu, \nu\in
\widehat{E}$ 
we have
\[
N_{\mu+\nu}= N_\mu * N_\nu, 
\]
where $N_\mu$ denotes  the measure on $\widehat{E}$ 
such that $\int g \, d N_\mu=Ng(\mu)$ for all $g\in bp\cb(\widehat{E})$. 
Note that if $N$ is a nonzero  branching kernel on $\widehat{E}$  then $N_{\bf 0}=\delta_{\bf 0}\in M(\widehat{E})$.

A right (Markov) process with state space $\widehat{E}$ is called \textit{branching process} 
provided that its transition function is formed by branching kernels. For the probabilistic
 interpretation of this analytic branching property see, e.g., \cite{Fi88}, page 337.\\

\noindent
\textbf {Multiplicative functions and branching kernels on $\widehat{E}$}.

For every real-valued, $\cb(E)$-measurable function
$\vf$  define the {\it multiplicative function }  
$\widehat{\varphi}:\widehat{E}\longrightarrow \R_+$ as
\begin{equation}\nonumber
\widehat{\varphi}({\bf x})=
\left\{
\begin{array}{l}
\displaystyle\prod_{k}\varphi(x_k), \textrm{ if }{\bf x}=(x_k)_{k\geq 1}\in \widehat{E},  {\bf x}\not= {\bf 0},  \\[2mm]

1, \;\;\;\;\; \textrm{ if } {\bf x}= {\bf 0},
\end{array}
\right. 
\end{equation}
(cf. \cite{Si68}; see also \cite{BeOp11}).
Observe that each multiplicative function $\widehat{\varphi}$, ${\varphi}
\in p{\mathcal B}(E)$, ${\varphi}\leq 1$, is the restriction to $\widehat{E}$ of an
exponential function on $M(E)$,
$$
\widehat {\varphi}= e_{-\ln {\varphi}}.
$$
In the harmonic analysis on configuration spaces the multiplicative function $\widehat{\varphi}$ is called {\it  coherent state}; see, e.g.,  
\cite{FiKoOl11}.

If $\widehat \vf$ is a
multiplicative function on $\widehat{E}$,  then  $\widehat\vf(\mu +\nu)=\widehat\vf(\mu)
\widehat\vf(\nu)$ for al $\mu, \nu \in \widehat{E}$ and therefore
$p_1  * p_2 (\widehat\vf)=p_1 (\widehat\vf)\  p_2 (\widehat\vf).$

Let $N$ be a sub-Markovian  kernel on $(\widehat{E},\cb(\widehat{E}))$. By Remark 3.1 in \cite{BeOp11}   
the following assertions are equivalent:

$(i)$ $N$ is a branching kernel.

$(ii)$ For all $\vf\in p\cb(E),$   $\vf  \leqslant  1,$
\[
N\widehat{\vf}=\widehat{ (N\widehat{\vf} )|_E  }.
\]

$(iii)$ $N$ maps multiplicative functions into multiplicative
functions.\\

It is possible to construct   branching kernels on $\widehat{E}$, using the above characterization, as follows:

\noindent
$(3.1)\quad$  For every sub-Markovian  kernel $B: p\cb(\widehat{E})\longrightarrow
p\cb(E)$ there exists a branching kernel $\widehat{B}$ on
$(\widehat{E},\cb(\widehat{E}))$ such that for every $\cb(E)$-measurable function
$\vf$, $|\vf|\leqslant  1$, we have
$$
\widehat{B}{\widehat{\vf}}= {\widehat{B{\widehat{\vf}}}}.
$$
The kernel $\widehat{B}$ is defined as:
$$
\widehat{B}_{{\bf x}}:=
\left\{
\begin{array}{l}
B_{ x_1}  * \ldots  *  B_{ x_n} ,\textrm{ if } {\bf x}=\delta_{x_1} + \ldots +\delta_{x_n} ,\; x_1, \ldots,x_n\in E,\\[2mm]

\delta_{\bf 0}\;\;\;\;\;\;\; \;\;\;\;\;\;\;\;\;\;\;\;\,,\textrm{ if } {\bf x}= {\bf 0}.
\end{array}
\right. 
$$
Conversely, if $H$ is a branching kernel on  $(\widehat{E},\cb(\widehat{E}))$
then there exists a unique sub-Markovian kernel $B:
p\cb(\widehat{E})\longrightarrow p\cb(E)$ such that $H=\widehat B$ (see Proposition 3.2 in \cite{BeOp11}).\\

\noindent
\textbf{Example of branching kernel on $\widehat{E}$}.
Let $q_k\in p\cb(E)$, $k\geqslant  1$,
satisfying $\sum_{k \geqslant  1}q_k \leqslant1$. 
Consider the kernel $B: p\cb(\widehat{E})\longrightarrow p\cb(E)$ defined as
\[
B g(x) :=\sum_{k \geqslant  1}q_k(x)g_k(x,\ldots , x),\;   g \in bp\cb(\widehat{E}),\, x\in E,\eqno{(3.2)}
\]
where $g_k:=g|_{E^{(k)}}$ for all $k\geqslant  1$.
By $(3.1)$ there exists a branching kernel $\widehat B$ on $\widehat{E}$ such that  for all
$\vf\in p\cb(E)$, $\vf\leqslant  1$,  we have
\[
\widehat B\widehat \vf|_E= \sum_{k\geqslant  1} q_k \vf^k.
\]

\section{Discrete branching processes} 

Let $X=(\Omega,\Fscr,\Fscr_t, X_t,\theta_t, P^x)$ be a fixed Borel
right process with space $E$ and suppose that its transition function $(T_t)_{t\geqslant  0}$ of $X$ is Markovian.

Let $B:bp\cb(\widehat{E})\longrightarrow bp\cb(E)$
be a sub-Markovian  kernel such that
\[
\sup_{x\in E} Bl_1(x)  < \infty.  \eqno{(4.1)}
\]
 
Note that $(4.1)$ is precisely condition $4.1.2$ from \cite{Si68}, $(2.1)$  in \cite{DaGoLi02}, or
$(4.25)$ from \cite{Li11}
and  if $B$ is given by $(3.2)$ then $(4.1)$ is equivalent with
\[
\sup_{x\in E} \sum_{k\geqslant  1} k q_k(x)< \infty.
\]

We also fix a function $c\in bp\cb(E)$ and we denote by $(T_t^{c})_{t\geqslant  0}$ 
the transition function of the process obtained by killing $X$ with
the multiplicative functional $(e^{-\int_0^t c(X_s)ds})_{t\geqslant  0}$.
It is given by the Feynman-Kac formula:
\[
T_t^{c} f(x)= E^x (e^{-\int_0^t c(X_s)\, ds} f(X_t) ), \quad  f\in bp\cb(E), \, x\in E . 
\]
Note that  if $\mathcal{L}$ is the infinitesimal generator of $X$, then the above 
killed process has the generator $\mathcal{L}-c.$ 

We denote by  $\cu=(U_\alpha)_{\alpha>0}$ the resolvent of the process $X$ and let 
$\cu^c=(U^c_\alpha)_{\alpha>0}$ 
be the resolvent of kernels induced by $(T_t^{c})_{t\geqslant  0}$, i.e.,
the resolvent family of the process killed with $c$. 

Denote by  $\cb_{\mbox{\textsf{u}}}$ the set of all functions $\vf\in p\cb(E)$ such
that $\vf \leqslant  1$.\\


Recall that a map $H: \cb_{\mbox{\textsf{u}}}
\longrightarrow \cb_{\mbox{\textsf {u}}}$ is called \textit{absolutely monotonic}
provided that there exists a sub-Markovian kernel $\mbox{\textbf {H}}:
bp\cb(\widehat{E}) \longrightarrow bp\cb(E)$ such that $H \vf = \mbox{\textbf{
H}}\widehat \vf$ for all $\vf\in \cb_{\mbox{\textsf {u}}}$. 
By $(3.1)$   we
have:\\[1mm]

We also have  (cf. Lemma 2.2 and Theorem 1 from \cite{Si68}): \\[1mm]

\noindent 
$(4.3)\quad$ \textit{ If $H, K$ are absolutely monotonic
then their composition $H K$ is also absolutely monotonic and $
\widehat{\mbox{\textbf{HK}}}= \widehat{\mbox{\textbf{H}}} \widehat{\mbox{\textbf{K}}}. $  
The map $H\longmapsto \widehat{\mbox{\textbf{H}}}$ is a
bijection between the set of all absolutely monotonic mappings and
the set of all branching kernels on $\widehat{E}$. }\\

In the next proposition we solve an appropriate integral equation
(following  the approach of \cite{Si68}, see also \cite{BeOp11}); 
we present its proof  in Appendix $(A3)$.

\begin{proposition}  \label{prop4.1}  
For any $\varphi\in \mathcal{B}_{\mbox{\textsf {u}}}$ the equation
\[
h_t(x)=T_t^{c}\varphi (x)+\int_{0}^t T^c_{t-u}(c
B\widehat{h_u})(x) du,\; t\geqslant  0,\; x\in E, \eqno{(4.4)}
\]
has a unique solution $(t,x)\longmapsto H_t \varphi(x)$ jointly measurable in $(t,x)$, such that
$H_t \varphi \in\mathcal{B}_{\mbox{\textsf {u}}} $ for all $t>0$ and the following assertions hold.

$(i)$ For each $t>0$ the mapping $\varphi\longmapsto H_t \varphi $ is
absolutely monotonic and it is Lipschitz  with the constant $\beta_o t$, where
\[
\beta_o:= \| c\|_\infty \| Bl_1 \|_\infty. 
\]

$(ii)$  The familiy
$(H_t)_{t\geqslant  0}$ is a  semigroup of (nonlinear) operators on
$\mathcal{B}_{\mbox{\textsf {u}}}$. 
If $B1=1$ then $H_t 1=1$ for all $t\geqslant 0$.

$(iii)$  For each  $x\in E$ the function  $t\longmapsto H_t \varphi(x)$ is 
right continuous on $[0,\infty)$,  provided that  $t\longmapsto T_t^{c} \varphi(x)$ is right
continuous.
\end{proposition}

\begin{remark} \label{rem4.2}  
$(i)$  If $\vf\in \cb_{\mbox{\textsf {u}}}$ and $t\geqslant  0$ then the sequence $(H_t^n\vf)_{n\geqslant  0}$ defined by $(A3.4)$ in Appendix 
converges uniformly to the solution  $H_t \vf$ of $(4.4)$.
The assertion is a consequence of  the following inequality which may be proved by induction:
\[
\| H_t^{n+1}\vf - H_t^n \vf \|_\infty \leqslant  \frac{(\beta_o t)^n}{n!} \| \vf \|_\infty \quad \mbox{ for all } n \geqslant  0.
\]

$(ii)$ Note that if  $\mathcal{L}$ is the infinitesimal generator of the base process $X$, 
then $(4.4)$ is formally equivalent to
\[
\frac{d}{dt} h_t = (\cl -c) h_t + cB\widehat{h_t }, \; t\geqslant  0.
\]

$(iii)$ If $B$ is given by $(3.2)$ then the condition $B1=1$ is equivalent with
\[
\sum_{k\geqslant  1} q_k(x)=1 \,  \mbox{ for all } x\in E.
\]
\end{remark} 

\begin{corollary} \label{cor4.3} 
For each $t\geqslant  0$ let  $\widehat{\mbox{\textbf{H}}_t}$ be the branching kernel on $\widehat{E}$ associated by $(4.2)$ 
with the absolutely monotonic operator ${H}_t$ from Proposition \ref{prop4.1},
$H_t \vf = \widehat{\mbox{\textbf{H}}_t}\widehat\vf|_E$ for all $\vf \in  \mathcal{B}_{\mbox{\textsf {u}}}$.
Then the following assertions hold.

$(i)$ The family $(\widehat{\mbox{\textbf{H}}_t})_{t\geqslant  0}$ is a sub-Markovian semigroup of branching kernels 
on $(\widehat{E}, \cb(\widehat{E}))$. 

$(ii)$ For each $t\geqslant  0$ and $f\in bp\cb(E)$
define the function $V_t f\in p\cb(E)$ as
\[
V_t f: =-\ln H_t(e^{-f}).
\]
Then the family $(V_t)_{t\geqslant  0}$ is a semigroup of (nonlinear) operators on $bp\cb(E)$ and 
\[
\widehat{\mbox{\textbf{H}}_t}(e_f)=e_{V_t f} \mbox{ for all } f\in bp\cb(E). \eqno{(4.5)}
\]
\end{corollary}
\begin{proof}
Assertion $(i)$  follows from $(4.3)$ and Proposition \ref{prop4.1}.
To prove assertion $(ii)$ it is enough to show that if  $f\in bp\cb(E)$ then $V_t f\in bp\cb(E)$. If $f \leqslant  M$, because 
$H_t (e^{-M}) \geqslant  T_t^{c}(e^{-M})$ and since the transition function of $X$ is Markovian, we have 
$V_t f\leqslant  \! V_t M \! \leqslant  \!\! -\ln(e^{-M} E^x ( e^{-\int_0^t c(X_s)ds}) ) \leqslant 
$ $M+t\|c\|_\infty.$
\end{proof}

\begin{remark} \label{rem4.4} 
$(i)$  If $v\in \cb_{\mbox{\textsf {u}}}$ is such that $\widehat{v}$ is an invariant function 
with respect to the branching semigroup $(\widehat{\mbox{\textbf{H}}_t})_{t\geqslant  0}$ 
on  $\widehat{E}$, then $v$ belongs to the domain of $\mathcal{L}$ 
(the infinitesimal generator of the base process $X$)
and
\[
(\cl -c) v +cB\widehat v =0.\eqno{(4.6)}
\]
In particular, if $B$ is given by $(3.2)$, then $v$
is the solution of the nonlinear equation
\[
(\cl -c)v +c\sum_{k\geqslant  1} q_k v^k=0.
\]
It turns out that $(\widehat{\mbox{\textbf{H}}_t})_{t\geqslant  0}$  is the main tool
in the treatment of the (nonlinear) Dirichlet problem associated
with the equation $(4.6)$;
see Proposition 6.1 and Subsection 6.2 in \cite{BeOp11} and Section 3 from  \cite{BeOp14}.

$(ii)$ 
The equation $(4.4)$ and the equality $(4.5)$ are analogous to
$(2.2)$ and respectively $(2.4)$ from \cite{DaGoLi02}, where, however, 
the forthcoming branching process having $(\widehat{\mbox{\textbf{H}}_t})_{t\geqslant  0}$ 
as transition function is used.
Assertion $(ii)$ of Corollary \ref{cor4.3}  shows that the semigroup approach for the continuous branching  
(developed by E.B. Dynkin \cite{Dy02} and P.J. Fitzsimmons \cite{Fi88}; see also \cite{Be11} and \cite{Li11}) 
is analogue to the above construction of the branching semigroup in the discrete branching case, 
due to N. Ikeda, M. Nagasawa, S. Watanabe, and M.L. Silverstein. 
Recall that in the case of the continuous branching
the family $(V_t)_{t\geqslant  0}$ is the so called ``cumulant semigroup''; for more details see Section 5 below.

$(iii)$ Assume that $E$ is an Euclidean open set,  $X$ is the Brownian motion on $E$, and $B$ is given by $(3.2)$. 
Then by $(4.5)$ and Remark \ref{rem4.2} $(ii)$ one can see that the cumulant semigroup $(V_t)_{t\geqslant  0}$ 
is formally the solution of the  
following evolution equation
\[
\frac{d}{dt} V_t f= \Delta V_t f - |\nabla V_t f |^2 +c(1-\sum_{k\geqslant 1} q_k e^{(1-k)V_t f}), \, t\geqslant 0.
\]
This should be compare with the equation satisfied by the cumulant semigroup  of a mesure-valued  continuous branching process
(cf.  $(2.2)'$ from \cite{Fi88} and $(5.1)$ below), in particular, in the case of the super-Brownian motion:
$\frac{d}{dt} V_t f= \Delta V_t f - (V_t f )^2, \, t\geqslant 0.$

$(iv)$ We refer to the survey  article \cite{BeLuOp12}
for a version of  assertion $(ii)$  of Corollary \ref{cor4.3} and for 
further connections between the continuous and 
discrete measure-valued processes.
\end{remark}

We establish now a linear version of Proposition \ref{prop4.1} which leads to a result on the perturbation
with kernels of the sub-Markovian semigroups and resolvents; 
see Appendix $(A4)$ for its proof.

\begin{proposition}\label{prop4.4}  
Let $K$ be a sub-Markovian kernel on $(E,\cb(E))$.
Then for any $f \in bp\mathcal{B}(E)$ the equation
\[
r_t(x)=T_t^{c} f (x)+\int_{0}^t T^c_{t-u}(cK{r_u})(x) du,\; t\geqslant  0,\; x\in E, \eqno{(4.7)}
\]
has a unique solution $Q_t f  \in bp \mathcal{B}(E)$,  the function  
$(t,x)\longmapsto Q_t f (x)$ is measurable and the following assertions hold.

$(i)$    The family $(Q_t)_{t\geqslant  0}$ is a  semigroup of sub-Markovian kernels on $(E, \mathcal{B}(E))$
and it is the transition function of a Borel right process with state space $E$.

$(ii)$ The function  $t\longmapsto Q_t f (x)$ is 
right continuous on $[0,\infty)$ for each  $x\in E$  if and only if the function  $t\longmapsto T_t^{c} f (x)$ has the same property.

$(iii)$  The resolvent of kernels  $\cu^o=(U^o_\alpha)_{\alpha>0}$ on $(E,\cb(E))$ induced by $(Q_t)_{t\geqslant  0}$
($U^o_\alpha =\int_0^\infty e^{-\alpha t} Q_t dt$, $\alpha>0$) 
has the property $(h1)$. 
If $\beta>0$ then $U^o_\beta =U^c_\beta+ U^c_\beta  cKU^o_\beta =U^c_\beta+ G_\beta   U^c_\beta$,
where $G_\beta$ is the bounded kernel defined as $G_\beta:= \sum_{k\geqslant  1} (U^c_\beta  cK)^k$.
We  have
$\ce(\cu^o_\beta)\subset  \ce(\cu^c_\beta)$,   $[b\ce(\cu^o_\beta)]=  [b\ce(\cu^c_\beta)]=[b\ce(\cu_\beta)]$,
and $G_\beta( \ce(\cu^c_\beta) )\subset \ce(\cu^o_\beta)$.
\end{proposition}

If $M\in \cb(E)$, $\beta>0$, and $u\in \ce(\cu_\beta^c)$  (resp. $u\in \ce(\cu_\beta^o)$), 
let ${}^{c}\! R^M_\beta u$ (resp. ${}^{o}\! R^M_\beta u$) be the reduced function of $u$ on $M$ 
with respect to $\mathcal{U}^c_\beta$ (resp. $\mathcal{U}^o_\beta$).
Let further $v_o:=U^o_\beta 1=U^c_\beta f_o$, where $f_o:=1+cKU^o_\beta 1$,
and fix a finite measure $\lambda$ on $(E, \cb(E))$.
We denote by $c_\lambda^c$ (resp. $c_\lambda^o$) the induced capacity:
$c^c_\lambda(M):=\inf\{\lambda({}^{c}\! R_\beta^D v_o): D \mbox{ open}, M\subset D\}$
(resp. $c^o_\lambda(M):=\inf\{\lambda({}^{o}\! R_\beta^D v_o): D \mbox{ open}, M\subset D\}$).\\

A result related to the following corollary  was stated in  Proposition 3.7 from \cite{BeTr11}.

\begin{corollary}  \label{cor4.5} 
We have $c^o_\lambda \leqslant  c^c_{\lambda'} \leqslant  c_{\lambda'}$, where $\lambda':= \lambda+\lambda\circ G_\beta$.
In particular, if the capacity $c_{\lambda'}$ is tight then  the capacities $c^c_{\lambda'}$ and $c^o_\lambda$ are also tight.
\end{corollary}

\begin{proof}
Because $\ce(\cu_\beta) \subset \ce(\cu_\beta^c)$ we get
${}^{c}\! R^M_\beta v_o\leqslant  {}^{c}\! R^M_\beta u_o\leqslant  R^M_\beta u_o$ for every  $M\in \cb(E)$
(where recall that
$u_o= U_\beta f_o\geqslant  v_o$) and therefore $c^c_{\lambda'} \leqslant  c_{\lambda'}$.
By assertion $(iii)$ of Proposition \ref{prop4.4} we may apply 
the result from  \cite{BeBo04}, Proposition 5.2.5, to obtain the inequality of kernels
${}^{o}\! R_\beta^M\leqslant   {}^{c}\! R_\beta^M + G_\beta {}^{c}\! R_\beta^M$
which leads to $c^o_\lambda(M)\leqslant  c^c_{\lambda'}(M)$.
\end{proof}

Let $\widehat{\cu}=(\widehat{U}_\alpha)_{\alpha>0}$ 
(resp.  $\widehat\cu^c=(\widehat{U}^c_\alpha)_{\alpha>0}$) 
be the sub-Markovian  resolvent of kernels on  $(\widehat{E}, \cb(\widehat{E}))$ generated by
the semigroup  $(\widehat{\mbox{\textbf{H}}_t})_{t\geqslant  0}$
(resp.  by $(\widehat{T}^c_t)_{t\geqslant  0}$).\\

Let further 
\[
\beta_1:= ||B l_1||_\infty.
\]
and assume that $B1=1$, hence $\beta_1\geqslant 1$. 
We suppose that $\beta_1>1$ and that the function $c$ is such that 
$ c < \frac{\beta_1}{\beta_1-1}$. 
Let  $(Q_t)_{t\geqslant  0}$ be the semigroup given by Proposition \ref{prop4.4}, 
with the sub-Markovian kernel $K$ on $(E,\cb(E))$ defined as
$Kf:= \frac{c}{c+\beta_1} B(l_f)$ and with $c+\beta_1$ instead of $c$. 

\begin{lemma} \label{lem4.5} 
If $B$ is given by $(3.2)$ and $c$ does not depend on $x\in E$, then
\[
Q_t f(x) = e^{-(c+\beta_1)  t}  E^x ( e^{\int_0^t cq_o(X_s)\, ds} f(X_t) ), \quad  f\in bp\cb(E), x\in E, t>0,
\]
where $q_o:= \sum_{k\geqslant  1} k q_k$, and we have $[b \ce(\cu^o)] = [b \ce(\cu_\beta)]$ for all $\beta>0$.

\end{lemma}

\begin{proof}
Observe first that in this case
\[
B(l_f)= q_o f,\,  f\in bp\cb(E),
\]
and equation $(4.7)$ with $c+\beta_1$ instead of $c$ and $Kf= \frac{c}{c+\beta_1} B(l_f)$  becomes
\[
k_t(x)= T_t  f (x)-\int_{0}^t T_{t-u}(b{k_u})(x) du,\; t\geqslant  0,\; x\in E,  
\]
where  $k_t:= e^{(c+\beta_1)t}r_t$ and  $b:= -cq_o$.
By Proposition 3.3 $(i)$ from \cite{Be11} the above equation has a unique solution $T^{b}_t f$ and
consequently $Q_t = T_t^{c+\beta_1- cq_o}$ for all  $t\geqslant 0$. 
The claimed expression of $Q_t$ follows now  also from \cite{Be11}, the equality $(3.3)$.
 Let $\beta_2:= c+\beta_1- c\beta_1$.  From Proposition 3.3 $(iii)$ from \cite{Be11},
 since $0<\beta_2\leqslant  c+\beta_1- cq_o \leqslant  \beta_1$, we have $T_t^{\beta_1}\leqslant  Q_t\leqslant  T_t^{\beta_2}$,
 $\ce(\cu_{\beta_2}) \subset  \ce(\cu^o )\subset  \ce(\cu_{\beta_1})$.
 By Remark \ref{rem2.1} we conclude that $[b \ce(\cu^o)]= [b \ce(\cu_{\beta_1})] = [b\ce(\cu_\beta)]$ for all $\beta>0$.
\end{proof}

In the next proposition  we prove relations between a set of excessive functions with respect to the base process $X$ 
and excessive functions with respect to the forthcoming branching process on $\widehat{E}$.
The key tool is the equality from the assertion (i) below.  
Note that a similar equality was obtained in the case of continuous branching processes in \cite{Fi88}, Proposition (2.7)
(see also \cite{Be11}, Proposition 4.2). 
It turns out that a perturbation of the generator of the base process was necessary in that case too,
however a simpler one, the perturbed semigroup $(Q_t)_{t\geqslant  0}$  being expressed with a Feynman-Kac formula, as in the particular case
discussed in  the above Lemma \ref{lem4.5}.

\begin{proposition}\label{prop4.7}  
The following assertions hold. 

$(i)$ If $f\in bp\cb(E)$ and $t>0$ then 
\[
e^{-\beta_1 t} \widehat{\mbox{\textbf{H}}_t} (l_f)= l_{Q_t f}.
\]

$(ii)$ If $\beta>0$ and $\beta':= \beta_1+\beta$ 
then the  following assertions are equivalent for every $u\in bp\cb(E)$.

$(ii.a)$ $u\in b\mathcal{E}(\cu^o_\beta)$.

$(ii.b)$  $l_u\in \mathcal{E}(\widehat\cu_{\beta'}).$ 

$(ii.c)$ For every $\alpha>0$ we have $1-e_{\alpha u}\in \mathcal{E}(\widehat\cu_{\beta'}).$ 

$(iii)$  If $u\in \ce(\cu^o_\beta)$ is a compact Lyapunov function on $E$ then 
$l_{1+u}\in \ce(\widehat\cu_{\beta'})$ is a compact Lyapunov function on $\widehat{E}$. 
\end{proposition}

\begin{proof}
$(i)$ We show first that if $f\in bp\cb(E)$ and $N$ is a kernel on  $\widehat{E}$ or from $\widehat{E}$ to $E$, such that
$N(l_1)$ is a bounded function then
\[
N(e_{\lambda f})'_{\lambda=0}:= \lim_{\lambda \to 0}  \frac{N(e_{\lambda f})- N1}{\lambda} = - N(l_f).  \eqno{(4.8)}
\]
Indeed, the assertion follows since  $\frac{1- e_{\lambda f}}{\lambda} \nearrow  l_f$ pointwise on $\widehat{E}$.
The Lipschitz property of $H_t$  (Proposition \ref{prop4.1} $(i)$)  implies
\[
||H_t(e^{-f})-1||_\infty  \leqslant  {\beta_o} t ||1-e^{- f}||_\infty \leqslant  {\beta_o} t ||f||_\infty,\quad
\left| \frac{H_t(e^{-\lambda f}) -1}{\lambda} \right|\leqslant  {\beta_o} t ||f||_\infty.
\]
Applying $(4.8)$ with  $N=\widehat{\mbox{\textbf{H}}_t}$ and since $\widehat{\mbox{\textbf{H}}_t}(e_{\lambda f})|_E= H_t(e^{-\lambda f})$
we deduce that $\widehat{\mbox{\textbf{H}}_t}(l_f)|_E \leqslant   {\beta_o} t ||f||_\infty$ and we claim that
\[
\widehat{\mbox{\textbf{H}}_t}(l_f)= -\widehat{\mbox{\textbf{H}}_t}(e_{\lambda f})'_{\lambda=0}= l_{\widehat{\mbox{\textbf{H}}_t}(l_f)|_E}. \eqno{(4.9)}
\]
By $(4.8)$ it is sufficient to show the second equality.
Let $\mu\in E^{(m)}$, $\mu=\sum_{k=1}^m \delta_{x_k}$. 
Using again $(4.8)$ for the kernel $\widehat{\mbox{\textbf{H}}_t}|_E$ and since $ \widehat{\mbox{H}_t} 1=1$ we get
\[
\widehat{\mbox{\textbf{H}}_t}(e_{\lambda f})'_{\lambda=0}(\mu)= 
(\prod_{k=1}^m \widehat{\mbox{\textbf{H}}_t}(e_{\lambda f})|_E (x_k) )'_{\lambda=0}=
\]
\[
(\widehat{\mbox{\textbf{H}}_t}(e_{\lambda f})|_E)'_{\lambda=0} (x_1) \cdot \widehat{\mbox{\textbf{H}}_t}(e_{ 0})(x_2)\cdot \ldots \cdot
\widehat{\mbox{\textbf{H}}_t}(e_{ 0})(x_m)+ \ldots= 
\]
\[
-[ \widehat{\mbox{\textbf{H}}_t}(l_f)(x_1)+\ldots + \widehat{\mbox{\textbf{H}}_t}(l_f)(x_m)]=
- l_{\widehat{\mbox{\textbf{H}}_t}(l_f)|_E}(\mu).
\]

Fot each  $t \geqslant  0$ define the function $\vf_t: \R_+ \longrightarrow bp\cb(E)$ by $\vf_t(\lambda):= V_t(\lambda f)$.
We clearly have $H_t(e^{-\lambda f})= e^{-\vf_t(\lambda)}$ and  from Proposition \ref{prop4.1} we obtain
\[
e^{-\vf_t(\lambda)}= T_t^{c}(e^{-\lambda f})+ \int_0^t T^c_{t-u}cB \widehat{\mbox{\textbf{H}}_t}(e_{\lambda f}) du, \quad t\geqslant  0.
\]
We have $\vf_t(0)=0$,  and using $(4.9)$   we get $\vf'_t(0)= \widehat{\mbox{\textbf{H}}_t}(l_f)|_E$.
By derivation of the above equation  in $\lambda=0$ and multiplying with $e^{-\beta_1 t}$, applying $(4.8)$ for 
$N=T^c_{t-u}cB \widehat{\mbox{\textbf{H}}_t}$, and again from $(4.9)$ we conclude that
\[
e^{-\beta_1 t} \vf'_t (0)=T^{c+\beta_1}_tf +\int_0^t T^{c+\beta_1}_{t-u} (c+\beta_1) K ({e^{-\beta_1 u} \vf'_u(0)}) du, \quad t\geqslant  0,
\]
where $Kf:= \frac{c}{c+\beta_1} B(l_f)$ and $(T^{c+\beta_1}_t)_{t\geqslant  0}$ is
the transition function of the process obtained by killing $X$ with
the multiplicative functional $(e^{-\beta_1 t -\int_0^t c(X_s)ds})_{t\geqslant  0}$, $T_t^{c+\beta_1}=e^{-\beta_1t} T_t^{c}$.
Hence $e^{-\beta_1 t} \vf'_t (0)$ verifies $(4.7)$ with $c+\beta_1$ instead of $c$ and the kernel $K$. 
Proposition \ref{prop4.4} implies  $e^{-\beta_1 t} \vf'_t (0)=Q_t f$ and by $(4.9)$
$e^{-\beta_1 t} \widehat{\mbox{\textbf{H}}_t}(l_f)= l_{e^{-\beta_1 t} \widehat{\mbox{\textbf{H}}_t}(l_f)|_E}= l_{Q_t f}.$\\

The proof of $(ii)$ is similar to that  of Corollary 4.3 from \cite{Be11}, using the above assertion $(i)$.

$(iii)$ Let $u\in \ce(\cu^o_\beta)$ be a Lyapunov function on $E$ and for each $n\in \N^*$ consider the compact set 
$K_n$ such that $[u\leqslant  n]\subset K_n$. Since $l_1=m$ on $E^{(m)}$, $m\geqslant  1$, we conclude that
$[l_{1+u}\leqslant  n]$ is included  in the compact set  $E^{(0)}\cup K_n \cup (K_n)^{(2)}\cup \ldots \cup (K_n)^{(n)}$ of $\widehat{E}$.
\end{proof}

Let $\ca:=\overline{[b\ce(\cu_\beta)]}$ ($=$ the closure in the supremum norm of $[b\ce(\cu_\beta)]$). 
Note that $\ca$ is an algebra; see, e.g. Corollary 23 from \cite{Be11}. 
By Remark \ref{rem2.1}  $\ca$ does depend on $\beta>0$
and using also Proposition \ref{prop4.4} $(iii)$ we get 
$\ca=\overline{[b\ce(\cu^c_\beta)]}=\overline{[b\ce(\cu^o_\beta)]}$.
Recall that $1-e^{-u} \in \ce(\cu^o_\beta)$ provided that $u\in \ce(\cu^o_\beta)$ (cf. Proposition \ref{prop4.7} $(ii)$) 
and therefore $\{ e^{-u}: u \in\ce(\cu^o_\beta) \} \subset \ca \cap \cb_{\mbox{\textsf {u}}}$.
We need a supplementary hypothesis:\\

\noindent
$(*)$ \textit{ 
There exist a countable subset $\cf_o$ of $b\ce(\cu^o_\beta)$ which is additive, $0\in \cf_o$,  and separates the finite measures on $E$ 
and a separable vector lattice $\mathcal{C}\subset \ca$  such that $\{ e^{-u}: u \in \cf_o \} \subset \mathcal{C}$ and 
$T_t^c \vf, T_t^{c}( cB\widehat\vf)\in  \mathcal{C}$ for all $\vf\in \mathcal{C}\cap \cb_{\mbox{\textsf {u}}}$ and $t>0$.
}

\begin{proposition} \label{prop4.8} 
The following assertions hold.

$(i)$ If  $c$  does  not depend on $x \in E$ and $B$ is given by $(3.2)$
with $\sum_{k\geqslant  1} ||q_k||_\infty  <\infty$,
 then $(*)$ is verified taking any countable subset $\cf_o$ of $b\ce(\cu^o_\beta)$ which is additive, $0\in \cf_o$,  
 and separates the finite measures on $E$, 
 and as $\mathcal{C}$  the  closure in the supremum norm of a separable vector lattice $\mathcal{C}_o\subset \ca$  
 such that $\{ e^{-u}: u \in \cf_o \} \subset \mathcal{C}_o$,
 $(q_k)_{k\geqslant  1}\!\! \subset \mathcal{C}_o$ and   $T_t (\mathcal{C}_o)\subset \mathcal{C}_o$ for all $t>0$.

$(ii)$   Assume that  $E$ is a locally compact space,  $(T_t^{c})_{t\geqslant  0}$  a $C_0$-semigroup on $C_0(E)$,  $c\in bC(E)$,
and  $B\widehat\vf$ and $B(l_\vf)$ also  belong to $ C_0(E)$  for all $\vf\in C_0(E) \cap \cb_{\mbox{\textsf {u}}}$.
Then $(*)$ holds taking $\mathcal{C}=C_0(E)\oplus \R$ and for any countable subset 
$\cf_o$ of $C_0(E) \cap b\ce(\cu^o_\beta)$ which is additive, $0\in \cf_o$,  and separates the  finite measures on $E$.

$(iii)$ If condition $(*)$ holds then $V_t (\cf_o) \subset \overline{\mathcal{C}}$
(the closure in the supremum norm of $\mathcal{C}$) for every $t\geqslant  0$.
\end{proposition}

\begin{proof}
By $(3.2)$ 
$B\widehat \vf= \sum_{k\geqslant  1} q_k \vf^k$ and $\mathcal{C}$ is an algebra.
Therefore $B\widehat\vf \in  \mathcal{C} \cap \cb_{\mbox{\textsf {u}}}$ provided that $\vf\in \mathcal{C} \cap \cb_{\mbox{\textsf {u}}}$,
so,  assertion $(i)$ holds.

Assertion $(ii)$ is clear, observing that $C_0(E)\oplus\R\subset \ca$. 
Note that by Remark \ref{rem4.2} $(i)$ we have $Q_t(C_0(E))\subset C_0(E)$ for all $t>0$ and using $(4.7)$
one can see that $(Q_t)_{t\geqslant  0}$  is also a $C_0$-semigroup on $C_0(E)$.

$(iii)$  Using condition $(*)$ it follows that $H_t^n (e^{-u})\in \overline{\mathcal{C}}\cap \cb_{\mbox{\textsf {u}}}$
 for all $n\geqslant  0$ and $u\in \cf_o$, where $H_t^n$ is given by $(A3.4)$.
Since the sequence $(H^n_t (e^{-u}))_n$ converges uniformly  (see  Remark \ref{rem4.2} $(i)$), 
$H_t (e^{-u})$ also belongs to $\overline{\mathcal{C}}$  which is an algebra and
we conclude that  $V_t u= -\ln  H_t (e^{-u}) \in \overline{\mathcal{C}}$.
\end{proof}

We state now  the main result of this paper,  the existence of a discrete branching process
associated with the base process $X$, the branching kernel $B$ an the killing kernel $c$.

\begin{theorem}\label{thm4.7} 
If  the base process $X$  is  standard
and  condition $(*)$ holds  then there exists a  branching standard process  with state space $\widehat{E}$,  
having $(\widehat{\mbox{\textbf{H}}_t})_{t\geqslant  0}$ as transition function.
\end{theorem}

\begin{proof}
According to $(2.1)$, in order to show that
$(\widehat{\mbox{\textbf {H}}_t})_{t\geqslant  0}$ is the  transition function of a 
c\`adl\`ag process with state space $\widehat{E}$,
 we  have to verify conditions $(h1)$-$(h3)$ for $\widehat{\cu}_{\beta'}$.

We show first that $(h1)$ is satisfied by  $\widehat\cu_{\beta'}$,
in particular, all the points of $\widehat{E}$ are non-branch points for $\widehat\cu_{\beta'}$.
We proceed as in the proof of Proposition 4.5 from \cite{Be11}.
According to Corollary 3.6 from \cite{St89}, it will be sufficient
to prove that the uniqueness of charges and the specific solidity
of potentials properties hold for
$\widehat{\mathcal{U}}_{\beta'}=(\widehat{U}_{\beta'+\alpha})_{\alpha> 0}$, 
where recall that  
$\widehat{U}_\alpha =\int_{0}^{\infty}e^{-\alpha t}\widehat{\mbox{\textbf {H}}_t} \;dt. $\\

\noindent
\textit{The uniqueness of charges property.}
  We have to show that if $\mu,\nu$ are two finite measures on $\widehat{E}$ such that $\mu\circ
\widehat{U}_{\beta'}=\nu\circ\widehat{U}_{\beta'}$  then $\mu=\nu.$ 
We get  $\mu(1)=\nu(1)$ and by Hunt's approximation theorem 
$\mu(F)=\nu(F)$ for every $F\in [b\ce(\widehat{\cu}_{\beta'})]$.
We already observed that the multiplicative  family of functions 
\[
\widehat\cf_o:=\{e_u : u\in \cf_o\}
\]
is a subset of $[b\ce(\widehat\cu_{\beta'})]$.
Therefore  $\mu(e_u)=\nu(e_u)$ for every $u\in \cf_o$
and $\cb(\widehat{E}) =\sigma(\widehat\cf_o)=\sigma(\widehat\ce(\cu_{\beta'}))$.
By a monotone class argument  we conclude that $\mu=\nu$.\\

\noindent
\textit{The specific solidity of potentials.}
We have to show that if $\xi, \mu\circ \widehat{U}_{\beta'} \in \mbox{\textsf {Exc}}( \widehat{\cu}_{\beta'})$ 
and $\xi\prec  \mu\circ \widehat{U}_{\beta'},$ then $\xi$ is a potential.
Here $\prec$ denotes the specific order relation on the convex cone $\mbox{\textsf {Exc}}( \widehat{\cu}_{\beta'})$ of all
$\widehat{\cu}_{\beta'}$-excessive measures:
if $\xi, \xi' \in \mbox{\textsf {Exc}}( \widehat{\cu}_{\beta'})$ then $\xi\prec \xi'$ means that there exists
$\eta\in \mbox{\textsf {Exc}}( \widehat{\cu}_{\beta'})$ such that $\xi+\eta=\xi'$.

Let  $\mathcal{A}_o$  be the additive semigroup generated by $\{ {V_tu}: u\in \cf_o, t\geqslant  0 \}$
and $[\widehat\ca_o]$ the
vector space spanned by $\{e_v: v \in \ca_o \}$. 
Then $[\widehat\ca_o]$ is an algebra of functions on $\widehat{E}$, $1 \in [\widehat\ca_o]$,  and since
$\widehat\cf_o \subset [\widehat\ca]$  we have $\sigma([\widehat\ca])=\cb(\widehat{E})$.
We prove now that
\[
[\widehat\ca_o]\subset  [b\ce(\widehat\cu_{\beta'})]. \eqno{(4.10)}
\]
Since $\widehat\cf_o\subset [b\ce(\widehat\cu_{\beta'})]$ we get from  $(4.5)$ that
$e_{V_tu}\in  [b\ce(\widehat\cu_{\beta'})] $ for all $u\in  \cf_o$.
By Corollary 2.3 from \cite{Be11} the vector space $[b\cs(\widehat\cu_{\beta'})] $ is an algebra
and therefore $e_v\in  [b\cs(\widehat\cu_{\beta'})] $ for all $v\in \ca_o$.
It remains to prove that the map 
$s \longmapsto \widehat{\mbox{\bf H}_s}  (e_v)(\mu)$ 
is right continuous on $[ 0, \infty )$ for every  $v\in  \ca_o$ and $\mu\in \widehat{E}$. 
Because $\widehat{\mbox{\bf H}_s}  (e_v)=\widehat{H_s(e^{-v})}$, 
we have  to prove the right continuity of the mapping
$s \longmapsto H_s (e^{-v})(x)$,  $x\in E$.
According to  Proposition \ref{prop4.1} $(iii)$
it will be sufficient to show that the map 
$s \longmapsto T^c_s  (e^{-v})(x)$ 
is right continuous for every $v \in \ca_o$ and $x\in E$.
This last right continuity property is satisfied 
since by Proposition \ref{prop4.8} $(iii)$ the function 
$e^{-v}$ belongs to the algebra $\ca$.

Let $\xi,\mu\circ\widehat  U_{\beta'  }\in \mbox{\textsf{
Exc}}(\widehat\cu_{\beta' })$, $\xi \prec\mu\circ\widehat
U_{\beta'  }$. We may suppose that $\mu(1)\leqslant  1$; if it
is not the case, then $\ds\mu=\sum_{n}\mu_n \ds$ with
$\mu_n(1)\leqslant  1$ for all $n$ and by Ch. 2 in \cite{BeBo04} there
exists a sequence $(\xi_n)_n \subset \mbox{\textsf {Exc}}(\widehat
\cu_{\beta' })$ such that $\ds \xi=\sum_{n} \xi_n \ds$ and
$\xi_n\prec \mu_n\circ \widehat{U}_{\beta}$ for every n. Let
$\varphi_{\xi}:\ce(\widehat\cu_{\beta' })\longrightarrow
\overline {{\R}}_{+}$ the functional defined by
$\varphi_{\xi}(F):=\widehat {L}_{\beta'  }(\xi,F)$,
$F\in\ce(\widehat \cu_{\beta'  })$, where $\widehat {L}_{\beta'}$ 
denotes the energy functional associated with
$\widehat\cu_{\beta'  }$.
By $(4.10)$ we may extend $\varphi_{\xi}$ to
an increasing linear functional on $[\widehat\ca_o]$. 
Let $\cm$ be the closure of $[\widehat\ca_o]$ with respect to the supremum norm.
Clearly, $\cm$ is a vector lattice and we claim that
$\varphi_{\xi}$ extends to a positive linear functional on $\cm$.
Indeed, if $(F_n)_n \subset [\widehat\ca_o]$ is a sequence
converging uniformly to zero and we consider a sequence
$(\nu_k\circ\widehat {U}_{\beta'  })_k \subset \mbox{\textsf
{Exc}}(\widehat\cu_{\beta' })$, $\nu_k\circ\widehat{U}_{\beta'} 
\nearrow \xi$, then we have
$|\varphi_\xi
(F_n)|=\sup_{k}|\nu_k(F_n)|\leqslant \liminf_{k}\nu_k(|F_n|)\leqslant  \varepsilon \liminf_{k}\nu_k(1)=
\varepsilon\widehat{L}_{\beta'
}(\xi,1)\leqslant \varepsilon\mu(1)\leqslant \varepsilon,
$
provided that $n\geqslant  n_0$ and $||F_n||_{\infty}<\varepsilon$ for all $n\geqslant  n_0$.
Since $\xi\prec \mu\circ\widehat{U}_{\beta'  }$ we have
$\varphi_{\xi}(F)\leqslant  \mu(F)$ for every $F \in \cm_{+}$.
Consequently,  if $(F_n)_n \subset \cm_+$ is a sequence decreasing
pointwise to zero, then $\varphi_{\xi}(F_n)\searrow 0$. 
By Daniell's  Theorem there exists a measure $\nu$ on $(\widehat{E},\cb(\widehat{E}))$ such that
$\varphi_{\xi}(F)=\nu(F)$ for all $F\in\cm$. 
In particular, if $u \in \cf_o$ then $\widehat{L}_{\beta'} 
(\xi,\widehat{\mbox{\textbf {H}}_t}(e_u))=\varphi_{\xi}(e_{V_tu})=\nu(\widehat{\mbox{\textbf {H}}_t}(e_u))$ 
and therefore
$\widehat{L}_{\beta'}(\xi,\widehat{U}_{\beta'}(e_u))= 
\lim_{k}\nu_k(\widehat{U}_{\beta'}(e_u))= 
\int_{0}^{\infty}{e^{-\beta' t}\lim_{k}\nu_k(\widehat{\mbox{\textbf {H}}_t}(e_u))dt}=
\int_{0}^{\infty}e^{-\beta'  t}\widehat{L}_{\beta'}(\xi,\widehat{\mbox{\textbf {H}}_t}(e_u))dt= 
\nu(\widehat{U}_{\beta'  }(e_u)).
$
We conclude that $\xi=\nu\circ\widehat{U}_{\beta'  }$.
Observe that actually we proved he following assertion:\\

\noindent
${(4.11)}$ \textit{ 
If $\xi, \eta$ are two finite measures from $\mbox{\textsf {Exc}}(\widehat{\cu}_{\beta'})$
and $\widehat{L}_{\beta'} (\xi, \widehat{\mbox{\textbf {H}}_t}(e_u))=
\widehat{L}_{\beta'} (\eta, \widehat{\mbox{\textbf {H}}_t}(e_u))$  for all   $u\in \cf_o$ and $t\geqslant  0$, then $\xi=\eta$.
}\\

Because $X$ is a Hunt process, Theorem $(47.10)$ in \cite{Sh88} implies  that $X$ 
has c\`adl\`ag trajectories in any  Ray topology (see $(A1.2)$ in Appendix).
Consider a Ray topology $\ct({\mathcal R})$ with respect to $\cu^{o}$, which is finer than the original
topology and  it is generated by a Ray cone ${\mathcal R}\subset b\ce(\cu^{o}_\beta)$ such that
$\cf_o\subset {\mathcal R}$.  So, without loosing the generality, 
we may assume in the sequel that the original topology of $E$ is a Ray one.

We check now condition $(h2)$.
Let $\lambda \in \widehat{E}$, $\lambda\neq {\bf 0}$,  and set as before $\lambda'=\lambda+ \lambda\circ G_\beta$.
Since the base process $X$ on $E$ has c\`adl\`ag trajectories
the capacity $c_{\lambda'}$ is tight (see Remark \ref{rem2.2})
and by Corollary \ref{cor4.5} the capacity
$c^o_\lambda$ is also tight. 
According to  $(2.2)$ and  Remark \ref{rem2.4} (i) there exists a compact Lyapunov function $u\in \ce(\cu^o_\beta)\cap L^1(\lambda)$.
Proposition \ref{prop4.7} $(ii)$ implies that
$l_u\in \ce(\widehat\cu_{\beta'})$ is a compact Lyapunov function on $\widehat{E}$ and $l_u(\lambda)<\infty$,
hence $(h2)$ holds.

We show that $(h3)$ also holds for $\widehat{\cu}_\beta$.
We take $l_1$ as the function $u_o$; observe  that by  Proposition \ref{prop4.7}  $(ii)$  
we have $l_1\in \ce(\widehat{\cu}_{\beta'})$ and clearly $l_1$
is a real-valued function.
Let $\mathcal{C}_o$ be a countable subset of  $b\ce(\cu^o_{\beta})$ such that  $\cf_o\subset \mathcal{C}_o$, $\mathcal{C}_o$ is additive,  
and  $p\mathcal{C}$ is included in the closure in the supremum norm of $( \mathcal{C}_o - \mathcal{C}_o)_{+}$.
Let further $\mathcal{R}_o$  be a countable dense subset of the Ray cone
$\mathcal{R}$ such that $\mathcal{C}_o\subset  \mathcal{R}_o$.
We may consider  $\widehat{\mathcal{R}}_o:=\{e_u : u\in \mathcal{R}_o \}$  as the countable set  $\cf$  from $(h2)$.
Note that since $\mathcal{R}_o$ generates the (Ray) topology on $E$, by Lemma 02 from \cite{INW68}
(see also the proof of Lemma 2.4 from \cite{BeOp11}), one can see that $\widehat{\mathcal{R}}_o$ generates the 
topology of $\widehat{E}$.
Let further $\xi, \eta$ be  two finite $\ce(\widehat{\cu}_{\beta'})$-excessive measures such that 
 $\widehat{L}_{\beta'} (\xi, e_u)= \widehat{L}_{\beta'} (\eta,e_ u)$ for all $u\in\mathcal{R}_o$ and
\[
\widehat{L}_{\beta'} (\xi+\eta, l_1)<\infty.  \eqno{(4.12)}
\]
To show that $\xi=\eta$ we can now proceed as in the proof of Theorem 4.9 from \cite{Be11}, Step I, page 699;
this procedure was also used in  the proof of Theorem 3.5 from \cite{BeLuOp12}.
Because  the $\sigma$-algebra $\cb(\widehat{E})$ is generated by the multiplicative family $\widehat\cf_o$, 
a monotone class argument implies that $\xi=\eta$ provided  that
\[
 \xi(e_u)= \eta(e_u) \mbox{ for all }  u\in \cf_o.
\]
By $(4.11)$  the above equality holds if
\[
\widehat{L}_{\beta'} (\xi, \widehat{\mbox{\textbf{H}}_t}(e_u))= 
\widehat{L}_{\beta'} (\eta, \widehat{\mbox{\textbf{H}}_t}(e_u)) \mbox{ for all }  u\in \cf_o \mbox{ and } t\geqslant  0. \eqno{(4.13)}
\]
From  $(4.12)$ and  $(A1.1.a)$  there exist two
measures $\mu$ and $\nu$ on $\widehat{E}_1$ such that $\xi=\mu\circ\widehat{U}^1_{\beta' }$ and 
$\eta=\nu\circ\widehat{U}^1_{\beta' }$. 
Let further $\widetilde{\mathcal{C}}_o:= \{\widetilde{e}_u \, :\,  u \in\mathcal{C}_o\}$. 
Because $\widetilde{\mathcal{C}}_o$ is a multiplicative class of functions on $\widehat{E}_1$ and 
$\mu(\widetilde{e}_u)=\widehat L_{\beta'}(\xi, e_u)= \widehat L_{\beta'}(\eta, e_u)=\nu(\widetilde{e}_u)$
for every $u \in \mathcal{C}_o$, by the monotone class
theorem we have
\[
\mu(F)=\nu (F) \mbox{ for all } F\in\sigma(\widetilde{\mathcal{C}}_o). \eqno{(4.14)}
\]
If $u\in \cf_o$ then by Proposition \ref{prop4.8} $(iii)$  there exists a
sequence $(f_n)_n \subset ( \mathcal{C}_o - \mathcal{C}_o)_{+}$ converging uniformly to $V_t u$.
Note that if $f \in ( \mathcal{C}_o - \mathcal{C}_o)_{+}$  then $e_f$  has a finely continuous extension
$\widetilde{e}_f$ from $\widehat{E}$ to $\widehat{E}_1$.  Since $e_u\in [b\ce(\widehat\cu_{\beta'})]$, using 
by $(4.5)$ we get that $e_{V_t u}$ belongs to $ [b\ce(\widehat\cu_{\beta'})]$ and by $(A1.1.c)$ 
it has a unique  extension
$\widetilde{e}_f$ from $\widehat{E}$ to $\widehat{E}_1$
(by fine continuity too).
As a consequence, and using $(A3.2)$, for every $\lambda \in \widehat{E}$ we have
$|e_{f_n}(\lambda) - e_{V_t u}(\lambda) |\leqslant  
||f_n- V_t u||_{\infty}\cdot
l_1(\lambda)$, hence $|\widetilde{e}_{f_n}-\widetilde{e}_{V_tu}|\leqslant 
||f_n-V_tu||_{\infty}\cdot \widetilde{l}_1$ on $\widehat{E}_1$. 
It follows that $(\widetilde{e}_{f_n})_n$ converges  pointwise to
$\widetilde{e}_{V_tu}$ on the set
$[\widetilde{l}_1<\infty]\in\sigma(\widetilde{\mathcal{C}}_o)$. 
From $(4.12)$ we get $(\mu+\nu)(\widetilde{l}_1)=\widehat {L}_{\beta }(\xi+\eta ,l_1)<\infty$, hence 
$\widetilde{l}_{1}<\infty$  $(\mu+\nu)$-a.e. 
Therefore, $1_{[\widetilde{l}_1<\infty]}\cdot\widetilde{e}_{V_tu}$ is
$\sigma(\widetilde{\mathcal{C}}_o)$-measurable and by $(4.14)$ we now deduce that
$\mu(\widetilde{e}_{V_tu})=\nu (\widetilde{e}_{V_tu})$ for all $u\in \cf_o.$ 
We conclude that $(4.13)$ holds, so $\xi=\eta$.
Applying $(2.1)$, $\widehat{\cu}$ is the resolvent of a standard process 
with state space $\widehat{E}$.

The quasi-left continuity follows by Lemma \ref{lem-quasileft}, taking $\widehat{\cf}_o$ as the multiplicative set $\ck$, 
since by Proposition \ref{prop4.1}  $(ii)$ the semigroup $(\widehat{\mbox{\textbf {H}}_t})_{t\geqslant  0}$  is Markovian.
\end{proof}

\begin{remark}  \label{rem4.11} 
$(i)$ For constructions of branching Markov processes we refer to
\cite{AsHe83}, \cite{INW68}, \cite{INW68a}, \cite{INW69}, and \cite{Si68}.
In particular, for locally compact base space, Theorem \ref{thm4.7} is
very closely related with Theorems 2.2--2.5 and Theorems 3.3--3.5  in \cite{INW68a}, where
the branching process is obtained by a path-wise piecing out procedure, 
starting from a canonical diagonal (branching) process on $\widehat{E}.$
Because in \cite{INW68a} it is not assumed any Feller or $(*)$ condition, 
it is of interest to show that the piecing out  procedure carry over to a Lusin base space. 
We thank the anonymous referee for suggesting us this comment.
Note that the diagonal (branching) process on $\widehat{E}$ was already essentially used in \cite{BeOp11} and \cite{BeOp14}, in the case of Lusin spaces.

$(ii)$ The extra point  ${\bf 0}$ is a trap for the branching process $X$ on $\widehat{E}$;  see also Theorem 1 from \cite{INW68}.
Indeed, the assertion follows since, with the notations from Proposition \ref{prop4.7},  
the mapping $l_1$ is $\widehat{\mathcal{U}}_{\beta'}$-excessive and we have $\{{\bf 0}\}= [l_1=0]$, so, the set $\{{\bf 0}\}$ is absorbing.
\end{remark}

\section{Application: continuous branching as base process} 
In this section we give an example of  a branching Markov process, 
having as base space the set of all finite configurations of positive finite measures on a topological space. 
Note that an example of branching type process  on this space  was given in \cite{BeOp11},
obtained by perturbing a diagonal semigroup with a branching kernel.

Let $Y$ be a standard (Markov) process with state space a Lusin
topological space $F$, called {\it spatial motion}. 
We fix a  {\it branching mechanism}, that is, a function $\Phi:F\times
[0,\infty)\longrightarrow {\R}$ of the form
$$
\Phi(x,\lambda)=-b(x)\lambda-a(x)\lambda^{2}+\int^{\infty}_{0}(1-e^{-\lambda
s}-\lambda s) N(x,ds)
$$
where $a {\geqslant} 0$ and $b$ are bounded $\cb(E)$-measurable functions
and  $N: {p}\cb((0,\infty))\longrightarrow {p}\cb(F)$ is a kernel
such that $N(u\land u^2)\in  {bp}\cb(F)$. Examples of branching
mechanisms are $\Phi(\lambda)=-\lambda^{\alpha}$ for $1<\alpha\leq
2$.

We first present the measure-valued branching Markov process associated with the
spatial motion  $Y$ and the branching mechanism $\Phi$, the $(Y,\Phi)$-{superprocess},
a standard process with state space $M(F)$, the space of all
positive finite measures on $(F,\cb(F))$, endowed with the weak
topology (cf. \cite{Fi88} and \cite{Dy02},
see also  \cite{Be11}).
For each $f \in {bp}\cb(F)$ the equation
$$
v_t{(x)}=P_t{f(x)}+\int_{0}^{t}{P_s(x,\Phi(\cdot, v_{t-s}))}ds,\ \
\ t\geqslant 0, \ \ \ x\in F,
$$
has a unique solution $(t,x)\longmapsto N_t f(x)$ jointly
measurable in $(t,x)$ such that $\ds \sup_{0 \leq s\leq
t}||v_s||_{\infty}$ $<\infty \ds$ for all $t>0$; we have denoted by
$(P_t)_{t\geqslant 0}$ the transition function of the spatial motion $Y$. 
Assume that $Y$ is conservative, that is $P_t1=1$.
The mappings $f \longmapsto N_t f$, $t\geqslant 0$, 
 form a nonlinear semigroup
of operators on ${bp}\cb(F)$ and the above equation is formally equivalent with
$$
\left\{
\begin{array}{l}
\frac{d}{dt}{v_t{(x)}}=\mbox{\sf L}  v_t(x)+ \Phi(x, v_t(x))\\[2mm]

v_0=f, 
\end{array}\nonumber
\right. \eqno{(5.1)}
$$
where $\mbox{\sf L}$ is the infinitesimal generator of the spatial
motion $Y$. For every $t \geqslant 0$ there exists a unique kernel
$T_t$ on $(M(F), \cm(F))$ such that
$$
T_t(e_f)=e_{N_t f}, \ \ \ f\in
{bp}\cb(F),  \eqno{(5.2)}
$$
where for a function $g\in bp\cb(F)$ the exponential function $e_g$ is defined on $M(F)$ as in Section 3.
Since the family $(N_t)_{t\geqslant  0}$ is a (nonlinear) semigroup on
${bp}\cb(F)$, $(T_t)_{t \geqslant  0}$ is a linear semigroup of kernels
on $M(F)$.
Suppose  that $F$ is a locally compact space,
$(P_t)_{t\geqslant 0}$ is a $C_0$-semigroup on $C_0(F)$, and  $a$, $b$, and $N$  do not depend on $x\in F$. 
We may assume  that $b\geqslant 0$.
Arguing as in the proof of Proposition 4.8 from \cite{Be11} one can see that
$N_t(C_0(F)) \subset C_0(F)$  and 
that $N_t(b\ce(\cu_{b'})) \subset b\ce(\cu_{b'})$ for every $t\geqslant 0$, 
where $b':=b+\beta$, with $\beta>0$.
Then $(T_t)_{t \geqslant  0}$ is the transition function of a standard
process with state space $M(F)$, called $(Y,\Phi)$-{\it
superprocess}; see, e.g.,  \cite{Fi88},  \cite{Be11}, and \cite{BeLuOp12}.
In addition, the  $(Y,\Phi)$-{superprocess}  is a branching process on $M(F)$, i.e.,
$T_t$ is a branching kernel on $M(F)$ for all $t>0$. 
Recall that the nonlinear semigroup $(N_t)_{t\geq 0}$ is called 
the {\it cumulant semigroup} of this branching process.

We can apply now the results from Section 4, starting with the
$(Y,\Phi)$-{superprocess} as base process with state space
$E:=M(F)$.

\begin{corollary} \label{cor:cont-disc} 
Let $c$ and $(q_k)_{k\geqslant 1}$ be positive real numbers such that  $\sum_{k\geqslant 1} q_k =1$,
$\sum_{k\geqslant 1} kq_k =:q_o<\infty$, and $0< \beta < c+q_o- cq_o$.
Then there exists a discrete branching process with state space $\widehat{M(F)}$, the set of all finite
configurations of positive finite measures on $F$, 
associated to $c$ and $(q_k)_{k\geqslant 1}$, and with base process the $(Y,\Phi)$-{superprocess}.
\end{corollary}

\begin{proof}
We apply Thoerem \ref{thm4.7}, so, we have to check condition $(*)$.
Let $\mathcal{R}$ be a Ray cone with respect to the resolvent $\cw=(W_\alpha)_{\alpha>0}$ of the process $Y$ on $F$,
constructed as in the proof of Proposition 4.8 from \cite{Be11},  $\mathcal{R}\subset b\ce(\cw_{b'})$,
such that $[\mathcal{R}\cap C_0(F)]$ is dense in $C_0(F)$.
Let $\mathcal{R}_o$ be a countable, additive, dense subset of 
$\mathcal{R}$.  
Then  $\{e_r:\, r\in \mathcal{R}_o\}$ is a multiplicative set of functions on $E$ and separates the measures on $E$.
Let further $\mathcal{C}$ be the closure in the supremum norm of the vector space spanned by 
 $\{e_w:\, w\in b\ce(\cw_{b'}) \cap C_0(F)\}$ and denote by $\cu= (U_\alpha)_{\alpha>0}$ the resolvent of
 the $(Y,\Phi)$-superprocess on $E$.
 By Corollary 4.4  from \cite{Be11} $1-e_w\in \ce(\cu_\beta)$ for all $w\in b\ce(\cw_{b'})$.
 Therefore $\mathcal{C}\subset \ca$ and  we may take as $\cf_o$ the additive semigroup generated by the set
 $\{1-e_w:\, w\in \mathcal{R}_o\}$. From $(5.2)$ and the above considerations 
 $T_t (\mathcal{C})\subset \mathcal{C}$ and since $\mathcal{C}$  is a Banach algebra 
 we clearly have $\{ e^{-u}:\, u\in \cf_o \}\subset \mathcal{C}$, hence condition $(*)$ holds.
\end{proof}

\section*{Appendix}

\noindent
\textbf{(A1) Excessive  measures, Ray cones.}
Let $\cu=(U_\alpha)_{\alpha\geqslant  0}$ be a sub-Markovian resolvent of kernels on $(E,\cb(E))$ such that condition $(h1)$ holds.

Let  $\ex$ be the set of all \textit{ $\cu$-excessive measures} on $E$:
$\xi\in \ex$ if and only if it is a $\sigma$-finite measure on
$(E,\cb(E))$ such that $\xi \circ \alpha U_\alpha \leqslant  \xi$ for all $\alpha>0$. 
Recall that if $\xi\in \ex$ then actually $\xi \circ \alpha U_\alpha \nearrow \xi$ as $\alpha \to \infty$.
We denote by $\pot$ the set of all \textit{ potential}
$\cu$-excessive measures: if $\xi\in \ex$ then $\xi\in \pot$ if
$\xi=\mu\circ U$, where $\mu$ is a $\sigma$-finite on $(E,\cb(E))$.

If $\beta>0$ then the \textit{energy functional} $L_{\beta}: \exb\times
\ce(\cu_\beta) \longrightarrow \overline{\R}_+$ is defined by
\[
L_{\beta} (\xi, u):= \sup \{ \nu(u) : \, \potb \ni \nu\circ
U_\beta\leqslant  \xi\}.
\]

\noindent
$(A1.1)$  There exists a second Lusin measurable space $(E_1,\cb_1)$ such that $E\subset E_1,\
  E\in\cb_1,\  \cb(E)=\cb_1|_E$,
and a resolvent of kernels $\cu^{1}=(U^{1}_{\alpha})_{\alpha>0}$
on $(E_1,\cb_1)$  satisfying $(h1)$ on this larger space,
 $U^{1}_{\beta}(1_{E_1 \setminus E})=0$, 
$\cu$ is the restriction of $\cu^{1}$ to $E$ 
(i.e.,  $U_{\beta}(g)=U^1_{\beta}(g^1)$, where $g^1\in \mbox{p}\cb_1$ and
$g^1|_{E}=g)$.
We clearly  have $ \mbox{\textsf {Exc}}({\cu_\beta})= \mbox{\textsf {Exc}}({\cu^1_\beta})$ and
the following property  holds (for one and therefore for all  $\beta>0$):\\

\noindent 
$(A1.1.a)\quad $ every $\xi\in \mbox{\textsf {Exc}}({\cu^1_\beta})$  with $L_{\beta}(\xi,1)<\infty$  
is a potential on $E_1$ (with respect to $\cu^1_\beta$).\\

\noindent
One can take  for $E_1$ the set of all extreme points of the set
$\{\xi\in \mbox{\textsf {Exc}}({\cu_\beta}) |\,  L_\beta(\xi,1)=1\}$,
endowed with the $\sigma$-algebra $\cb_1$ generated by the
functionals $\widetilde{u}$, $\widetilde{u}(\xi):=L_\beta(\xi,u)$
for all $\xi \in E_1$ and $u\in \ce(\cu_\beta)$. 
Let $(E' ,\cb' )$ be a Lusin measurable space such that $E \subset E' $, $E \in \cb' $, $\cb(E)=\cb' |_{E}$,
and there exists a proper sub-Markovian resolvent of kernels $\cu' =(U' _{\alpha})_{\alpha > 0}$ on $(E' ,\cb' )$
with $D_{\cu' _{\beta}}=E' $, $\sigma(\ce(\cu' _{\beta}))=\cb' $, $U' _{\beta}(1_{E'  \setminus E})=0$, $E' $
satisfies $(A1.1.a)$ with respect to $\cu' $,  and $\cu$ is the restriction of $\cu' $ to $E$.
Then the map $x\longmapsto \varepsilon_x \circ U' _{\beta}$ is a measurable isomorphism between
$(E' ,\cb' )$ and the measurable space $(E_{1},\cb_{1})$.

\noindent
\textbf{Extension of excessive functions from $E$ to $E_1$. }
If  $\xi=\mu\circ U_\beta\in \potb$ and $u\in \ce(\cu_\beta)$ then  by Theorem 1.4.5 from \cite{BeBo04} we have
\[
\leqno{(A1.1.b)}\quad\quad    L_\beta(\xi, u)= \int u \, d\mu.
\]

\noindent
$(A1.1.c)\quad$ 
For every  $u\in \ce(\cu_\beta)$ we consider the function 
$\widetilde{u}: E_1\longrightarrow \overline{\R}_+$ defined above,
\[
\widetilde{u}(\xi):=L_{\beta}(\xi, u), \quad \xi\in E_1.
\]
Then by $(A1.1.b)$ we have $\widetilde{u}(\varepsilon_x\circ U_\beta)= u(x)$ 
for all $x\in E$
and therefore, by the embedding of $E$ in $E_1$,
\[
\widetilde{u}|_E=u.
\]
In addition, $\widetilde{u}$ is $\cu_\beta^1$-excessive and it is the (unique) extension by fine continuity of 
$u$ from $E$ to $E_1$.\\

\noindent
$(A1.2)$ \textbf{Ray cones.} If $\beta>0$ then a \textit{Ray cone} associated with ${\mathcal
U}_\beta$ is a cone $\mathcal R$ of bounded $\mathcal
U_\beta$-excessive functions such that: $U_\alpha(\mathcal
R)\subset \mathcal R$ for all $\alpha>0$, $U_\beta\bigl((\mathcal
R-\mathcal R)_+\bigr)\subset \mathcal R$, $\sigma(\mathcal
R)=\cb(E)$, it is min-stable, separable in the supremum norm
and $1\in\mathcal R$. Such a Ray cone always exists. 
Below if we say Ray cone it is always meant
to be associated with one fixed resolvent $\mathcal U_\beta$.
If $\cu$ is transient (i.e., there exists a strictly positive function $f_o\in bp\cb(E)$ with $\sup_\alpha U_\alpha f_o<\infty$, 
then one can take $\beta=0$, that is, there is  a Ray cone of $\cu$-excessive functions).
A  \textit{Ray topology} on $E$ is a topology generated by a Ray cone; for
more details see ch. 1 in \cite{BeBo04} and also \cite{BeBoRo06a} for the
non-transient case.\\

\noindent
\textbf{(A2) Proof of Lemma \ref{lem-quasileft}.} 
As we already mentioned, we follow the classical approach, cf., e.g.,
page 48 from  \cite{Sh88}, page 115 in \cite{MaRo92}, see also pages 133-134 in \cite{BeBo04}
and the proof of Theorem 5.5 $(ii)$ from \cite{BeRo11b}.

We start with the construction of a convenient compactification of $E$, as in the proof of Theorem 5.2 from \cite{BeRo11b}.

Let $K$ be the compactification of $E$ with respect to  $\cf$.  
Since for every real-valued function    $u\in \ce(\cu_\beta^o)$ 
the real-valued process $(e^{-\beta t} u\circ {X}_t )_{t\geqslant  0}$ 
is a  right continuous (${P}^x$-integrable) 
supermartingale under ${P}^x$ for all $x\in E$,
it follows that this process has left limits ${P}^x$-a.s.
and we conclude that  ${X}$ has left limits in $K$ a.s.

Let  $(T_n)_n$ be an increasing sequence of stopping times
and $T=\lim_{n} T_n.$ 
It is no loss of generality to assume that $T$ is bounded.
From the above considerations the limit $Z:= \lim_n X_{T_n}$ exists in $K$ a.s. and $Z(\omega)\in E$ if $T(\omega)< \zeta(\omega)$.

In order to prove that $Z=X_T$ a.s.  on  $[T<\zeta]$, it is enough to show that for every $x\in E$
and $G\in bp\cb(K\times K)$,
\[
E^x ( G 1_{E\times E}  (Z, X_T) ) = E^x ( G 1_{E\times E}  (Z, Z) ) . \eqno{(A2.1)}
\]
Indeed, taking as $G$ the indicator function of the diagonal of $K\times K$, from $(A2.1)$ we get
$P^x([Z\in E, Z\not = X_T])=0$.

Note that every function $f$ from  $\overline{[\cf]}$ 
has an extension by continuity from $E$ to $K$, denoted by $\overline{f}$.
Since $[b\ce(\cu_\beta)]$ is an algebra, we may assume that $\cf$ is multiplicative.
In order to prove $(A2.1)$ we  first use the strong Markov property 
(clearly, $\overline{f}(Z)\in \cf_{T}$) and then the  $P^x$-a.s. equality
$\lim_n f(X_{T_n}) p_t g(X_{T_n})= \overline{f}(Z) \overline{p_t g}(Z)$ 
(because we take $f\in  [\cf]$ and $p_tg$  belongs to $\overline{[\cf]}$ provided that $g\in \mathcal{K}$):
\[
E^x ( \overline{f}(Z)  U_\alpha g(X_T) )= 
E^x ( \overline{f}(Z)E^{X_T}\!\!\! \int_0^\infty \!\!\! e^{-\alpha t} g (X_t)\d t )=
E^x ( \overline{f}(Z) e^{\alpha T} \int_T^\infty  \!\!\! e^{-\alpha t} g (X_t)\d t )=
\]
\[
\lim_n E^x ( f(X_{T_n}) e^{\alpha T_n} \int_{T_n}^\infty \!\!\! e^{-\alpha t} g (X_t)\d t )=
\lim_n E^x (  f(X_{T_n})  U_\alpha g(X_{T_n}) ) =
\]
\[
\lim_n E^x ( f(X_{T_n})  \int_0^\infty e^{-\alpha t} p_t g(X_{T_n})\d t )=
E^x ( \overline{ f}(Z)  \int_0^\infty e^{-\alpha t} \overline{p_t g}(Z)\d t ).
\]
By a monotone class argument we have for all $h\in bp\cb(K)$
\[
E^x ( h 1_E (Z)  U_\alpha g(X_T) )= 
E^x  ( h 1_E (Z)  \int_0^\infty e^{-\alpha t} \overline{p_t g}(Z)\d t )
\]
and therefore
\[
E^x ( h 1_E (Z)  U_\alpha g(X_T) )= 
E^x  (  h (Z)  \int_0^\infty e^{-\alpha t} p_t g (Z)\d t ; Z\in E  ) =E^x (  h (Z)  U_\alpha g(Z); Z\in E ).
\]
Because $\lim_{\alpha\to \infty} \alpha U_\alpha g= g$ 
(since $g$ is continuous), multiplying by $\alpha$ and letting $\alpha$  tend to infinity we get
\[
E^x ( h 1_E (Z)  g(X_T) )= E^x (  (h  g1_E )(Z) ).
\]
Using again monotone class arguments we obtain first
\[
E^x ( h 1_E (Z) \cdot  k1_E (X_T) ) = E^x ( h1_E  (Z) \cdot  k1_E (Z) ) \quad \mbox{ for all } h,k \in bp\cb(K),
\]
and then  $(A2.1)$.

If the transition function $(p_t)_{t\geqslant  0}$ is Markovian then 
the limit $Z= \lim_n X_{T_n}$ exists in $E$ a.s. 
Therefore, in this case it is enough to show that 
$(A2.1)$ holds for every 
$G\in bp\cb(E\times E)$.
Note that the extensions  by continuity of $f$ and $p_t g$  
from $E$ to $K$ are not longer necessary, in particular,
$\lim_n f(X_{T_n}) p_t g(X_{T_n})= {f}(Z) {p_t g}(Z)$  $P^x$-a.s. 
$\hfill\square$\\[2mm]

\noindent
\textbf{(A3) Proof of Proposition \ref{prop4.1}.}
Let
\[
K\varphi:=B\widehat{\varphi}.
\]
With this notation $(4.4)$  becomes
\[
h_t(x)=T_t^{c}\varphi(x)+\int_{0}^{t}T^c_{t-u}cKh_u (x)\;du, \quad
t\geqslant  0, \;  x\in E. \eqno{(A3.1)}
\]

We prove first the uniqueness.  
As in \cite{Si68}, the inequality $(4.11)$,  one can  see
that if $\vf, \psi\in \cb_{\mbox{\textsf {u}}}$ and $\mu\in \widehat{E}$ then
\[
|\widehat \vf(\mu)- \widehat\psi(\mu)| \leqslant  l_1(\mu)|| \vf -
\psi||_\infty . \eqno{(A3.2)}
\]
From $(4.1)$ and the $(A3.2)$ we conclude that 
\[
 \leqno{(A3.3)}\quad  \mbox{ the mapping } \vf \longmapsto cK\vf  \mbox{ is Lipschitz  with the constant } \beta_o.
\] 
If  $h_t$ and $h'_t$ are  two solutions of $(4.4)$ then for all $t\geqslant  0$
\[
\| h_t-h_t'\|_\infty \leqslant  \!\! \displaystyle\int_{0}^{t}  \|   T^c_{t-u}(\mid cK h_u- cK
h'_u\mid) \|_\infty du\leqslant  {\beta_o} \!\!  \displaystyle\int_{0}^{t}
\parallel h_u-h'_u
\parallel_\infty du.
\]
It follows by Gronwall's Lemma that $\parallel
h_t-h_t'\parallel_\infty=0$.\\

To prove the existence, define inductively the operators $H^n_t$, $n\geqslant  0$, as
$H_t^0 \varphi :=T_t^{c}\varphi,$
\[
H_t^{n+1}\varphi :=T_t^{c}\varphi+\displaystyle\int_{0}^{t}T^c_{t-u} cK H^{n}_{u}\varphi\; du,\;\varphi\in
\mathcal{B}_{\mbox{\textsf {u}}}. \eqno{(A3.4)}
\]
Clearly  the function $(t,x)\longmapsto  H^n_t \varphi(x)$ is measurable. 
We claim that the sequence $(H_t^n \varphi)_n$ is increasing. 
Indeed,  $H_t^1 \varphi=T_t^{c}\varphi +\displaystyle\int_{0}^{t}T^c_{t-u} cKH^0_u\varphi\; du\geqslant 
H_t^0\varphi$.
If we suppose that 
$H_t^{n-1}\varphi \leqslant  H_t^{n}\varphi$ then 
$H_t^{n+1} \varphi = T_t^{c} \varphi
+ \displaystyle\int_{0}^{t}T^c_{t-u}cKH_u^n\varphi \;du \geqslant 
T_t^{c}\varphi +\displaystyle\int_{0}^{t}T^c_{t-u}cKH_u^{n-1}\varphi\; du = H^n_t\varphi$. 
The last inequality holds because if  $\varphi\leqslant \psi$ then $cK\varphi\leqslant  cK\psi$ and in addition 
one can prove 
inductively that $H_t^n \varphi \leqslant  H_t^n \psi$ for all $n$.

We claim now that 
\[
H_t^n 1\leqslant  1 \quad \mbox{ for all } n\geqslant  0. \eqno{(A3.5)}
\]
We proceed again by induction. 
The inequality holds for $n=1$ because $(T_t^{c})_{t\geqslant  0}$
is sub-Markovian and $H_t^0 1=T_t^{c} 1$.
If we assume that $H_t^n 1\leqslant  1$ then $\widehat{H_t^n 1}\leqslant  1$  and therefore
\[
H_t^{n+1} 1 = T^c_t 1 + \int_0^t T^c_u(cB \widehat{H_t^n 1} ) du \leqslant   T_t^{c} 1 + \int_0^t T^c_u c\, du=
\]
\[
 E^x ( e^{-\int_0^t c(X_u)\, du}+ \int_0^t e^{-\int_0^s c(X_u)\, du} c(X_s)\, ds ) =1.
\]

If $\varphi \in \cb_{\mbox{\textsf {u}}}$ then by $(A3.5)$
$H_t^n \varphi \in \cb_{\mbox{\textsf {u}}}$ for all $n\geqslant  0$.
For $x\in E$, $t\geqslant  0$,  and $\varphi\in \mathcal{B}_{\mbox{\textsf {u}}}$ we set
\[
H_t\varphi(x):=\displaystyle\sup_{n}H^n_t\varphi(x).
\]
The function $(t,x)\longmapsto H_t \varphi (x)
$ is measurable, by $(A3.5)$ we have $H_t 1\leqslant  1$, $H_t(\cb_{\mbox{\textsf{u}}})\subset \cb_{\mbox{\textsf{u}}}$, 
and passing to the pointwise limit in $(A3.4)$ it follows that 
$(H_t
\varphi)_{t\geqslant  0}$ verifies $(A3.1)$.

$(i)$ We show inductively  that for all $n$  the operator $H^n_t$ is absolutely monotonic.
If $n=1$ then 
$H^1_t \varphi =T_t^{c}\varphi=\mbox{\textbf {T}}_t\widehat{\varphi}$, where
$\mbox{\textbf {T}}_t: bp\mathcal{B}(\widehat{E})\longrightarrow bp\mathcal{B}(E)$ 
is the kernel defined by
$\mbox{\textbf {T}}_tg=T_t^{c}(g|_E)$ for all  $g\in
bp\mathcal{B}(\widehat{E}).$ 
Hence
$H_t^1 \varphi=\mbox{\textbf {T}}_t \widehat{\varphi}$ for all $\varphi\in\mathcal{B}_{\mbox{\textsf{u}}}$ and therefore
$H^1_t$ is absolutely monotonic.
Suppose now that $H^n_t$ is absolutely monotonic, $H_t^n\vf= \mbox{\textbf {H}}_t^n \widehat{\vf}.$ 
We have
\[
H_t^{n+1}\varphi=\mbox{\textbf {T}}_t \widehat{\varphi}+ \int_{0}^{t}T_{t-u}c B\widehat{H_u^n \varphi}\;du=
(\mbox{\textbf {T}}_t + \displaystyle\int_{0}^{t}T_{t-u}c
B\widehat{\mbox{\textbf {H}}_u^n}\;du)\widehat{\varphi},
\]
where $\widehat{\mbox{\textbf {H}}_u^n}$ is the branching kernel on $\widehat{E}$ associated by $(4.2)$  with $\mbox{\textbf {H}}_u^n$. 
Taking 
\[
\mbox{\textbf {H}}_t^{n+1}:=
\mbox{\textbf {T}}_t + \int_{0}^{t}T_{t-u}c
B\widehat{\mbox{\textbf {H}}_u^n}\;du, \eqno{(A3.6)}
\]
it follows that  $H_t^{n+1}$ is also absolutely monotonic.
One can deduce from $(A3.6)$ that for all $t\geqslant  0$ the sequence of kernels
$(\mbox{\textbf {H}}_t^n)_{n\geqslant  0}$ is increasing and therefore we may  consider the kernel 
$\mbox{\textbf {H}}_t$ defined as $\mbox{\textbf {H}}_t:= \sup_n \mbox{\textbf {H}}_t^n$.
From the above considerations for all $\vf\in\mathcal{B}_{\mbox{\textsf{u}}}$ we have
$H_t\varphi = \sup_n H_t^n \varphi=\sup_n \mbox{\textbf {H}}_t^n \widehat{\vf}=
\mbox{\textbf {H}}_t\widehat{\varphi}$ and we conclude that  $H_t$ is absolutely  monotonic.

We prove now the Lipschitz property of the mapping $\vf\longmapsto H_t \vf$. 
For, if  $\vf, \psi\in \mathcal{B}_{\mbox{\textsf{u}}}$ and  $t\geqslant  0$ then 
by $(A3.1)$ and $(A3.3)$ 
\[
|| H_t \vf - H_t \psi ||_\infty\leqslant 
||\vf - \psi ||_\infty +  {\beta_o} \int_0^t  || H_u \vf - H_u \psi ||_\infty du
\]
and by Gronwall's Lemma we conclude that $|| H_t \vf - H_t \psi ||_\infty\leqslant  {{\beta_o}} t ||\vf - \psi ||_\infty$.

$(ii)$ The semigroup property of $(H_t)_{t\geqslant  0}$  is a consequence of the uniqueness. 
Indeed, we have to show that $H_{t'+t}\varphi =H_t(H_{t'}\varphi)$,
so, it is enough to prove that the mapping
$t\longmapsto H_{t'+t}\varphi$ verifies $(A3.1)$ with $H_{t'}\varphi$ instead of $\varphi$.
We have 
\[
H_{t'+t}\varphi= 
T_t^{c} T_{t'}\varphi + \int_{0}^{t'} \!\! T_t^{c}(T^c_{t'-u} cK {H_u \varphi})\, du +
\int_{t'}^{t'+t}\!\!\!  T^c_{t'+t-u} cK {H_u \varphi}\, du= 
\]
\[
T_t^{c}(T^c_{t'}\varphi+\int_{0}^{t'}\!\! T^c_{t'-u} cK {H_u
\varphi}\;du)+ \int_{0}^{t}\!\! T^c_{t-s} cK {H_{t'+s} \varphi}\, ds=
T_t^c H_{t'}\varphi +\int_{0}^{t}\!\! T^c_{t-s} cK {H_{t'+s}\varphi}\, ds.
\]

Suppose now that $B1=1$ and 
define inductively the operators $H'^n_t$, $n\geqslant  0$, as
$H'^0_t \varphi :=T_t^{c}\varphi+ \int_0^t T^c_u cK\vf du $,
\[
H'^{n+1}_t \varphi :=T^c_t\varphi+\displaystyle\int_{0}^{t}T^c_{u} cK H'^{n}_{t-u}\varphi\; du,\;\varphi\in
\mathcal{B}_{\mbox{\textsf{u}}}. \eqno{(A3.7)}
\]
We already observed that 
$ T^c_t 1 + \int_0^t T^c_u c\, du=1$, therefore $H'^0_t 1=1$ and by induction we get
that $H'^n_t 1=1$ for all $n\in \N$. Using $(A3.3)$ as before we obtain
\[
||H'^{n+1}_t\vf - H'^n_t \vf ||_\infty \leqslant  {\beta_o}\int_0^t ||H'^{n}_u \vf - H'^{n-1}_u \vf ||_\infty du
\]
and because 
$||H'^{1}_t\vf - H'^0_t \vf ||_\infty \leqslant  {\beta_o} ||\vf ||_\infty \int_0^t (2+{\beta_o} u)du =
$$  ||\vf ||_\infty (2{\beta_o} t + \frac{({\beta_o} t)^2}{2})$  again by induction
\[
||H'^{n+1}_t \vf - H'^n_t  \vf ||_\infty \leqslant   ||\vf ||_\infty \left
( 2 \frac{\,\, ({\beta_o} t)^{n+1}}{(n+1)!} +  \frac{\,\, ({\beta_o} t)^{n+2}}{(n+2)!}\right).
\]
Consequently, if $t_o>0$ is fixed then
\[
\sup_{\stackrel{x\in E}{t\leqslant  t_o}}\left| H'^{n+1}_t \vf(x) - H'^n_t  \vf(x) \right|\leqslant  
\left( 2 \frac{\,\, ({{\beta_o}} t_o)^{n+1}}{(n+1)!} +  \frac{\,\, ({\beta_o} t_o)^{n+2}}{(n+2)!}\right).
\]
It follows that the sequence $(H'^n_t\vf)_n$ is Cauchy in the supremum norm and passing to the limit in $(A3.7)$,
we deduce that the  pointwise limit of this sequence verifies $(A3.1)$, hence it is $H_t\vf$ by the uniqueness of the solution. 
In particular, $H_t 1=\lim_n H'^n_t 1=1$.

$(iii)$  Because the family $(H_t)_{t\geqslant  0}$ is a semigroup, it is enough to prove the right continuity in $t=0$. 
Since $H_t \varphi(x)$ is a solution of $(A3.1)$ and the function
$u\longmapsto  T^c_{t-u} cKh_u(x)$ is bounded on $[0,\infty)$, by dominate convergence we get
$\lim_{t\searrow 0} \int_{0}^{t}T^c_{t-u} cKh_u (x)\,du=0$, hence $t\longmapsto H_t\varphi(x)$ is
right continuous in $t=0.$
$\hfill\square$\\[2mm]

\noindent
\textbf{(A4) Proof of Proposition \ref{prop4.4}.} 
We may suppose that $f\leqslant  1$ and define the sub-Markovian kernel 
$\mbox{\textbf {K}}: bp\mathcal{B}(\widehat{E})\longrightarrow bp\mathcal{B}(E)$ by
$\mbox{\textbf {K}} g:= K(g|_E)$ for all  $g\in bp\mathcal{B}(\widehat{E}).$ 
We apply Proposition \ref{prop4.1}  for  $B:= \mbox{\textbf {K}}$
and observe  that  $H^n_t$ extends to kernel on $(E, \cb(E))$  for each  $t\geqslant  0$ and $n\in \N$.  
Since the limit of $(H^n_t)_n$   is increasing we conclude that the solution of the
equation $(4.7)$ also extends to  a kernel $Q_t$ on  $(E, \cb(E))$. 
This proves  $(ii)$ and the first part of assertion $(i)$. 

$(iii)$  The equality $U^o_{\beta} =U^c_\beta+ U^c_{\beta} cK U^o_{\beta}$ 
follows from $(4.7)$ by a straightforward calculation.
Then by induction $U^o_{\beta} =U^c_\beta+  (U^c_{\beta} cK)^2 U^c_\beta+\ldots +  (U^c_{\beta} cK)^n U^c_\beta+
 (U^c_{\beta} cK)^{n+1} U^o_\beta$ and letting $n$ tends to infinity we have $U^o_{\beta} =U^c_\beta+ G_\beta U^c_{\beta}$.
The kernel $G_\beta$ is bounded because
$U^c_{\beta} cK1\leqslant $$ ||\frac{c}{c+\beta}||_\infty
\lim_{t\to\infty}\int_0^t T_u^{c+\beta} (c+\beta) du\leqslant $$ 
\frac{c_o}{c_o +\beta},$
where $c_o:= ||c||_\infty$.
If $u\in b\ce(\cu^o_{\beta})$ then clearly  $\alpha U^c_{{\beta}+\alpha} u\leqslant  u$ for all $\alpha>0$
because  $U^c_\alpha\leqslant  U^o_\alpha$. 
From $\lim_{t\to 0} Q_t u=u$ we get by $(ii)$
that $\lim_{t\to 0} T_t^{c} u=u$, hence $u\in \ce(\cu^c_{\beta})$, 
$b\ce(\cu^o_{\beta})\subset \ce(\cu^c_{\beta})$.
The inequality  $U^c_\beta \leqslant  U^o_\beta$ for all $\beta>0$
implies that the function $G_\beta U^c_\beta f=U^o_\beta f - U^c_\beta f$ is $\cu^o_\beta$-excessive
for every $f\in bp\cb(E)$.
If $v\in b\ce(\cu^c_{\beta})$ then we take a sequence $(f_n)_n\subset bp\cb(E)$ such that
$U^c_\beta f_n \nearrow v$ and therefore $G_\beta U^c_\beta f_n \nearrow G_\beta v\in \ce(\cu^o_{\beta})$,
$U^o_\beta f_n \nearrow v+G_\beta v \in b\ce(\cu^o_{\beta})$, so $v\in [b\ce(\cu^o_{\beta})].$
We clearly have $\ce(\cu_{\beta})\subset \ce(\cu^c_{\beta})\subset \ce(\cu_{c_o+\beta})$
and by Remark \ref{rem2.1} we get $[b\ce(\cu_{\beta})]=  [b\ce(\cu_{c_o+\beta})]= [b\ce(\cu^c_{\beta})]$.

We check now $(h1)$ for $\cu^o$. 
From $[b \ce(\cu^o_{\beta})]= [b \ce(\cu^c_{\beta})]$  and since $\cu^c$ verifies $(h1)$
we conclude that $\cb(E) =\sigma(\ce(\cu^c_\beta))=\sigma(\ce(\cu^o_\beta)).$
The constant function $1$ is $\cu^o$-supermedian and it belongs to $[b \ce(\cu^o_{\beta})]$, therefore
$\lim_{t\to 0} Q_t 1=1$, $1\in \ce(\cu^o)$.

The fact that $(Q_t)_{t\geqslant  0}$ is the transition function of a right Markov process with state space $E$ 
is a consequence of Proposition 5.2.4, Proposition 3.5.3, and Corollary 1.8.12   from \cite{BeBo04}.
$\hfill\square$

\vspace{5mm}

\noindent
{\bf Notes added in proof.} 
Investigating  the  branching properties of the solution of a fragmentation equation for the mass distribution, 
in  \cite{BeDeLu15} it is used the main result of this paper (Theorem \ref{thm4.7})
for   constructing a branching process corresponding to a rate of loss of mass greater than a given
strictly positive threshold.

\bibliographystyle{amsplain}

\end{document}